\documentclass{amsart}
\usepackage{bbm}
\usepackage{hyperref}

\newcommand{\C}{\mathbbm C}

\newcommand{\R}{\mathbbm R}

\newcommand{\N}{\mathbbm N}

\newtheorem{theorem}{Theorem}
\newtheorem{proposition}[theorem]{Proposition}
\newtheorem{lemma}[theorem]{Lemma}
\newtheorem{corollary}[theorem]{Corollary}

\theoremstyle{definition}

\newtheorem{definition}[theorem]{Definition}

\theoremstyle{remark}
\newtheorem{remark}[theorem]{Remark}

\numberwithin{equation}{section}
\numberwithin{theorem}{section}

\newcommand{\coloneqq}{\,\raise0.08ex\hbox{\textnormal{:}}\!\!=}

\def\XXint#1#2#3{{\setbox0=\hbox{$#1{#2#3}{\int}$}
     \vcenter{\hbox{$#2#3$}}\kern-.5\wd0}}

\begin{document}

\title[Plateau flow]
{Plateau flow\\ or\\ the heat flow for half-harmonic maps}
\author{Michael Struwe}
\address[Michael Struwe]{Departement Mathematik\\ETH-Z\"urich\\CH-8092 Z\"urich}
\email{michael.struwe@math.ethz.ch}

\date{\today}

\begin{abstract}
Using the interpretation of the half-Laplacian on $S^1$ as the 
Dirichlet-to-Neumann operator for the Laplace equation on the ball $B$, we devise 
a classical approach to the heat flow for half-harmonic maps from $S^1$ to a 
closed target manifold $N\subset\R^n$, recently studied by 
Wettstein, and for arbitrary finite-energy data we obtain a result fully analogous
to the author's 1985 results for the harmonic map heat flow of surfaces and in similar 
generality. When $N$ is a smoothly embedded, oriented closed curve $\Gamma\subset\R^n$
the half-harmonic map heat flow may be viewed as an alternative gradient flow 
for a variant of the Plateau problem of disc-type minimal surfaces.
\end{abstract}

\maketitle

\section{Background and results}
\subsection{Half-harmonic maps and their heat flow}
Let $N\subset\R^n$ be a closed sub-manifold, that is, compact and without boundary. 
The concept of a half-harmonic map $u\colon S^1\to N\subset\R^n$ was introduced by 
Da Lio-Rivi\`ere \cite{Da Lio-Riviere-2011}, who together with Martinazzi in 
\cite{Da Lio-Martinazzi-Riviere-2015}, Theorem 2.9, also made the interesting observation that 
the harmonic extension of a half-harmonic map yields a free boundary minimal surface
supported by $N$, a fact which also was noticed by Millot-Sire \cite{Millot-Sire-2015}, 
Remark 4.28. 

In his PhD-thesis, Wettstein \cite{Wettstein-2022},\cite{Wettstein-2021a}, \cite{Wettstein-2021b},
recently studied the corresponding heat flow given by the equation
\begin{equation}\label{1.1}
   d\pi_N(u)\big(u_t+(-\Delta)^{1/2}u\big)=0\ \hbox{ on } S^1\times[0,\infty[,
\end{equation}
where $u_t=\partial_tu$, 
and where $\pi_N\colon N_{\rho}\to N$ is the smooth nearest neighbor projection on a 
$\rho$-neighborhood $N_{\rho}$ of the given target manifold to $N$, 
and, with the help of a fine analysis of the fractional differential operators involved,
he showed global existence for initial data of small energy.

Moser \cite{Moser-2011} and Millot-Sire \cite{Millot-Sire-2015} contributed important results
to the study of half-harmonic maps by exploiting the fact that for any smooth $u\colon S^1\to\R^n$
we can represent the half-Laplacian classically in the form 
\begin{equation}\label{1.2}
   (-\Delta)^{1/2}u=\partial_rU
\end{equation}
where $U\colon B\to\R^n$ is the harmonic extension of $u$ to the unit disc $B$
\footnote{The classical formula \eqref{1.2} is a special case of a much more general 
result, due to Caffarelli-Silvestre \cite{Caffarelli-Silvestre-2007}, who pointed out that many 
nonlocal problems involving fractional powers of the Laplacian can be related to a local,
possibly degenerate, elliptic equation via a suitable extension of the solution to a half-space.}.
Here, using the identity \eqref{1.2} we are able to remove the smallness assumption in
Wettstein's work and show the existence of a ``global'' weak solution to the 
heat flow \eqref{1.1} for data of arbitrarily large (but finite) energy, 
which is defined for all times and smooth away from finitely many ``blow-up points''
where energy concentrates, and whose energy is non-increasing. The solution is unique in this
class in exact analogy with the classical result \cite{Struwe-1985} by the author on the 
harmonic map heat flow for maps from a closed surface to a closed target manifold 
$N\subset\R^n$; see Theorem \ref{thm1.2} below.

In order to describe our work in more detail, let 
\begin{equation*}
   H^{1/2}(S^1;N)=\{u\in H^{1/2}(S^1;\R^n);\;u(z)\in N\hbox{ for almost every }z\in S^1\}.
\end{equation*}
Interpreting $S^1=\partial B$, where $B=B_1(0;\R^2)$ and tacitly identifying a map 
$u\in H^{1/2}(S^1;N)$ with its harmonic extension 
$U\in H^1(B;\R^n)$, for a given function $u_0\in H^{1/2}(S^1;N)$ we then seek to find a 
family of harmonic functions $u(t)\in H^1(B;\R^n)$ with traces $u(t)\in H^{1/2}(S^1;N)$
for $t>0$, solving the equation 
\begin{equation}\label{1.3}
   d\pi_N(u)\big(u_t+\partial_ru\big)=u_t+d\pi_N(u)\partial_ru=0\ \hbox{ on } S^1\times[0,\infty[,
\end{equation}
with initial data
\begin{equation}\label{1.4}
   u|_{t=0}=u_0\in H^{1/2}(S^1;N).
\end{equation}

\subsection{Energy}
The half-harmonic heat flow may be regarded as the heat flow for the half-energy 
\begin{equation*}
   E_{1/2}(u)=\frac12\int_{S^1}|(-\Delta)^{1/4}u|^2d\phi
\end{equation*}
of a map $u\in H^{1/2}(S^1;N)$. Note that the half-energy of $u$ equals the standard 
Dirichlet energy
\begin{equation*}
   E(u)=\frac12\int_B|\nabla u|^2dz
\end{equation*}
of its harmonic extension $u\in H^1(B;\R^n)$. Indeed, integrating by parts we have
\begin{equation}\label{1.5}
   \int_B|\nabla u|^2dz=\int_{S^1}u\partial_ru\,d\phi=\int_{S^1}u(-\Delta)^{1/2}u\,d\phi
   =\int_{S^1}|(-\Delta)^{1/4}u|^2d\phi,
\end{equation}
where we use the Millot-Sire identity \eqref{1.2} and where 
the last identity easily follows from the representation of the operators
$(-\Delta)^{1/2}$ and $(-\Delta)^{1/4}$ in Fourier space with symbols $|\xi|$, $\sqrt{|\xi|}$,
respectively, and Parceval's identity. 
\footnote{Conversely, via Fourier expansion we also can prove \eqref{1.5} directly.
Computing the first variations of $E$ and $E_{1/2}$, respectively, we then 
obtain \eqref{1.2}.}
Therefore, in the following for convenience we may always 
work with the classically defined Dirichlet energy. Moreover, we may interpret the 
half-harmonic heat flow as the heat flow for the Dirichlet energy in the class of harmonic 
functions with trace in $H^{1/2}(S^1;N)$;
see Section \ref{Energy} below for details.

\subsection{Results}
Identifying $\R^2\cong\C$, we denote as $M$ the $3$-dimensional M\"obius group 
of conformal transformations of the unit disc, given by
\begin{equation*}
   M=\{\Phi(z)=e^{i\theta}\frac{z-a}{\bar{a}+z}\in C^{\infty}(\bar{B};\bar{B}):\ 
   |a|<1,\ \theta\in\R\}.
\end{equation*}
Observe that the Dirichlet energy is invariant under conformal transformations, and we 
have $E(u)=E(u\circ\Phi)$ for any $u\in H^1(B;\R^n)$ and any $\Phi\in M$.

For smooth data we then have the following result. 

\begin{theorem}\label{thm1.1} 
Let $N\subset\R^n$ be a closed, smooth sub-manifold of $\R^n$, and suppose that the 
normal bundle $T^{\perp}N$ is parallelizable. Then the following holds:

i) For any smooth $u_0\in H^{1/2}(S^1;N)$ there exists a time $T_0\le\infty$ and a unique 
smooth solution $u=u(t)$ of \eqref{1.3}, hence of \eqref{1.1}, with data \eqref{1.4} 
for $0<t<T_0$.

ii) If $T_0<\infty$, we have concentration in the sense that for some $\delta>0$ and any
$R>0$ there holds
\begin{equation*}
   \sup_{z_0\in B,\,0<t<T_0}\int_{B_R(z_0)\cap B}|\nabla u(t)|^2dz\ge\delta,
\end{equation*} 
and for suitable $t_k\uparrow T_0$ there exist finitely 
many points $z_k^{(1)},\dots,z_k^{(i_0)}$
and conformal maps $\Phi_k^{(i)}\in M$ with $z_k^{(i)}\to z^{(i)}\in\bar{B}$ and
$\Phi_k^{(i)}\to\Phi_{\infty}^{(i)}\equiv z^{(i)}$ weakly in $H^1(B)$
such that $u(t_k)\circ\Phi_k^{(i)}\to\bar{u}^{(i)}$ weakly in $H^1(B)$ as $k\to\infty$,
where $\bar{u}^{(i)}$ is non-constant and conformal and satisfies 
\begin{equation}\label{1.6}
   d\pi_N(\bar{u}^{(i)})\partial_r\bar{u}^{(i)}=0,\, 1\le i\le i_0.
\end{equation}
Moreover, there exists $\delta=\delta(N)>0$ such that $E(\bar{u}^{(i)})\ge\delta$,
and $i_0\le E(u_0)/\delta$. 
Finally, $u(t_k)$ smoothly converges to a limit $u_1\in H^{1/2}(S^1;N)$
on $\bar{B}\setminus\{z^{(1)},\dots,z^{(i_0)}\}$. 

iii) If $T_0=\infty$, then, 
as $t\to\infty$ suitably, $u(t)$ smoothly converges to a half-harmonic limit map
away from at most finitely many concentration points where non-constant half-harmonic 
maps ``bubble off'' as in ii).
\end{theorem}

By the Da Lio-Rivi\`ere interpretation of \eqref{1.6},
the ``bubbles'' $\bar{u}^{(i)}$ as well as the limit  $u_{\infty}$ of the 
flow conformally parametrize minimal surfaces with free boundary on $N$, 
meeting $N$ orthogonally along their free boundaries.

The hypothesis regarding the target manifold $N$ in particular is fullfilled if $N$ is 
a closed, orientable hypersurface of co-dimension $1$ in $\R^n$, or if $N$ is a
smoothly embedded, closed curve $\Gamma\subset\R^n$.

It would be interesting to find examples of initial data for which the flow blows
up in finite time, as in the work of Chang-Ding-Ye \cite{Chang-Ding-Ye-1992}
on the harmonic map heat flow.

For data in $H^{1/2}(S^1;N)$ the following global existence result holds, which is our 
main result.

\begin{theorem}\label{thm1.2} 
For $N\subset\R^n$ as in Theorem \ref{thm1.1} the following holds:
i) For any $u_0\in H^{1/2}(S^1;N)$ there exists a unique global weak solution of 
\eqref{1.3} with data \eqref{1.4} as in Definition \ref{def6.3}, 
whose energy is non-increasing and which is smooth for positive time away from finitely 
many points in space-time where non-trivial 
half-harmonic maps ``bubble off'' in the sense of Theorem \ref{thm1.1}.ii). 

ii) As $t\to\infty$ suitably, $u(t)$ smoothly converges to a half-harmonic limit map
away from at most finitely many concentration points where non-constant half-harmonic 
maps ``bubble off'' as in Theorem \ref{thm1.1}.iii).
\end{theorem}

Note that uniqueness is only asserted within the class of partially regular weak solutions with 
non-increasing energy, as in the case of the harmonic map heat flow. It would be interesting
to find out if the latter condition suffices, as in the work of Freire \cite{Freire-1995},
\cite{Freire-1996},
and, conversely, to explore the possibility of ``backward bubbling'' in \eqref{1.3}, as in
the examples of Topping \cite{Topping-2002} for the latter flow.

\subsection{Key features of the proof and related flow equations}
In our approach, in a similar vain as Lenzmann-Schikorra \cite{Lenzmann-Schikorra-2020},
we uncover and exploit surprising regularity properties of the normal 
component $d\pi_N^{\perp}(u)\partial_ru$ for the harmonic extension of $u$, likely related to the 
fractional commutator estimates for the normal projection in the work of
Da Lio-Rivi\`ere \cite{Da Lio-Riviere-2011} or the regularity estimates of 
Da Lio-Pigati \cite{Da Lio-Pigati-2020}, Mazowiecka-Schikorra \cite{Mazowiecka-Schikorra-2018}, 
and others. 

The use of the Dirichlet-to-Neumann map for the
harmonic extension $u\colon B\to\R^n$ of $u$ instead of the half-Laplacian, 
and the simple identity \eqref{3.2} as well as equation \eqref{3.5} allow to perform the 
analysis using only local, classically defined operators, avoiding fractional calculus almost 
entirely.

Note that equation \eqref{1.3} is similar to the equation governing the (scalar) evolution 
problem for conformal metrics $e^{2u}g_{\R^2}$ of prescribed geodesic boundary curvature
and vanishing Gauss curvature on the unit disc $B$, studied for instance by 
Brendle \cite{Brendle-2002} or Gehrig \cite{Gehrig-2020}. 
In contrast to the latter flows, due to the presence of the projection operator mapping $u_r$
to its tangent component, the flow \eqref{1.3} at first sight appears to be  degenerate. 
However, surprisingly, within our framework we are able to obtain similar smoothing properties 
as in the case of the harmonic map heat flow of surfaces. 

A different heat flow associated with half-harmonic maps, 
using the half-heat operator $(\partial_t-\Delta)^{1/2}$ instead of 
\eqref{1.1}, was suggested by Hyder et al. \cite{Hyder-et-al-2021}, and they obtained global
existence of partially regular, but possibly non-unique weak solutions for their flow, with a
possibly large singular set of measure zero. 

\subsection{Applications to Plateau problem}
In the case when $N$ is a smoothly embedded, oriented closed curve $\Gamma\subset\R^n$
the half-harmonic heat flow \eqref{1.3} may furnish an alternative gradient flow 
for the Plateau problem of minimal surfaces of the type of the disc, which has a long and 
famous tradition in geometric analysis.

Having been posed by Plateau in the 1890's, Plateau's problem 
was finally solved independently by Douglas \cite{Douglas-1931} and Rad\'o \cite{Rado-1930}
in 1930/31.
 In order to analyse the set of {\it all} minimal surfaces solving the Plateau problem, 
including saddle points of the Dirichlet integral, thereby building on Douglas' ideas, 
in 1939 Morse-Tompkins \cite{Morse-Tompkins-1939} proposed 
a critical point theory for Plateau's problem in the sense of Morse \cite{Morse-1937}, 
attempting to characterize non-minimizing solutions
as ``homotopy-critical'' points of Dirichlet's integral. 
However, in the 1980's Tromba \cite{Tromba-1984},
\cite{Tromba-1985} pointed out that it was not  
even clear that all smooth, non-degenerate minimal surfaces would be ``homotopy-critical'' 
in the sense of Morse-Tompkins \cite{Morse-Tompkins-1939}.
To overcome this problem, Tromba developed a version of degree theory that could be applied
in this case and which yielded at least a proof of the ``last'' Morse inequality,
which is an identity for the total degree. 

In 1982, finally, this author \cite{Struwe-1984} 
recast the Plateau problem as a variational problem on a closed convex set and he
was able to develop a version of the Palais-Smale type critical point theory for the 
problem within this frame-work, which allowed him to obtain all Morse inequalities in a 
rigorous fashion; see the monograph \cite{Struwe-1988} and the paper by Imbusch-Struwe
\cite{Imbusch-Struwe-1999} for further details. In the papers \cite{Struwe-1986} by this
author and \cite{Jost-Struwe-1990} by Jost-Struwe 
the approach was extended to the case of multiple boundaries and/or higher genus.

A key element of critical point theory for a 
variational problem is the construction of a pseudo-gradient flow for the problem at hand. 
In \cite{Struwe-1984} this was achieved in an ad-hoc way. However, starting with the work of 
Eells-Sampson \cite{Eells-Sampson-1964} on the harmonic map heat flow, it is now an 
established approach in geometric analysis to study the 
(negative) ($L^2$-)gradient flow related to a variational problem, 
similar to the standard heat equation. 
For Plateau's problem, such a flow was obtained by Chang-Liu \cite{Chang-Liu-2005} 
within the frame-work laid out by Struwe \cite{Struwe-1984} 
in the form of a parabolic variational inequality, for which Chang-Liu obtained a solution 
of class $H^2$ by means of a time-discrete minimization scheme. Rupflin \cite{Rupflin-2017},
Rupflin-Schrecker \cite{Rupflin-Schrecker-2018}
studied the analogous parabolic variational inequality in the case of an annulus,
which again had previously been studied by this author \cite{Struwe-1986} by means of an ad-hoc 
pseudo-gradient flow. 

In view of the much better regularity properties of the flow equation \eqref{1.3} it would be 
tempting to regard this as the correct definition of the canonical gradient flow for the 
Plateau problem, but an important issue still needs to be addressed.

\subsection{Monotonicity}
Recall that in the classical Plateau problem $u(t)$ is required to induce a (weakly) monotone 
parametrization of $\Gamma$ for each $t>0$. Even though it may seem likely that -- at least 
for curves $\Gamma$ on the boundary of a convex body in $\R^3$ -- this Plateau 
boundary condition will be preserved along the flow \eqref{1.3} whenever it is satisfied initially,
at this moment even for a strictly convex planar 
curve $\Gamma\subset\R^2$ it is not clear whether this actually happens.
However, the results that we obtain also seem to be of interest if we drop the Plateau condition.
In particular, our results motivate the study of smooth minimal surfaces with continous trace
covering only a part of the given boundary curve $\Gamma$; dropping the monotonicity condition 
also brings the parametric approach to the Plateau problem closer to the approach via
geometric measure theory or level sets. 

\subsection{Plateau flow}
It should be straightforward to extend our results to the case when the disc $B$ is replaced 
by a surface $\Sigma$ of higher genus with boundary $\partial\Sigma\cong S^1$, 
if for given initial data $u_0\in H^{1/2}(S^1;N)$ we consider a family $u=u(t)$ 
in $H^{1/2}(S^1;N)$ solving the equation \eqref{1.3}, that is,
\begin{equation*}
   u_t+d\pi_N(u)\partial_{\nu}u=0
\end{equation*}
instead of \eqref{1.1}, where for each time we harmonically extend $u(t)$ to $\Sigma$ and 
denote as $\partial_{\nu}u$ the outward normal derivative of $u$ along $\partial\Sigma$,
as was proposed and analysed by Da Lio-Pigati \cite{Da Lio-Pigati-2020} in the 
time-independent case.
Similarly, one might study the flow \eqref{1.3} on a domain $\Sigma$ with multiple boundaries. 
Of course, in order for the flow to converge to a minimal surface in the case of higher genus or
higher connectivity it will be necessary to couple the flow \eqref{1.3} with a corresponding
evolution equation for the conformal structure on $\Sigma$, as in the work of
Rupflin-Topping \cite{Rupflin-Topping-2019} on minimal immersions.
Note that on a general domain $\Sigma$ the flow equations \eqref{1.1} and \eqref{1.3} no 
longer agree. In order to clearly distinguish the flow equation \eqref{1.3} from the 
equation \eqref{1.1} defining the half-harmonic map heat flow, we therefore 
propose to say that \eqref{1.3} defines the ``Plateau flow''.

\subsection{Outline}
After a brief discussion of energy estimates in Section \ref{Energy}, in 
Section \ref{A regularity estimate} we present the analytic core of the argument for 
higher regularity in Section \ref{Higher regularity} and for the blow-up analysis, 
later presented in Section \ref{Blow-up}.
These tools are  also instrumental in proving uniqueness of partially regular weak solutions in 
Section \ref{Uniqueness}. 
The $L^2$-bounds for higher and higher derivatives which we establish in 
Section \ref{Higher regularity}, assuming that energy does not concentrate,
may be of particular interest. These bounds either concern estimates 
for $\nabla\partial^k_{\phi}u$ on $B$ or on $\partial B$, and we view the latter bounds
as stronger by an order of $1/2$. These bounds may be used interlaced, as we later do in 
Section \ref{Weak solutions}, to prove uniform smooth estimates, locally in time, for 
smooth flows with smooth initial data converging in $H^{1/2}(N;S^1)$. Since the latter 
data are dense in $H^{1/2}(N;S^1)$ we thus not only obtain existence of weak solutions
for arbitrary data $u_0\in H^{1/2}(N;S^1)$ but also can show their smoothness for positive
time and hence are able to derive Theorem \ref{thm1.2} from Theorem \ref{thm1.1}. 
A peculiar feature is that one set of regularity estimates can only be 
obtained globally, that is on all of $B$, whereas the other set of estimates may be 
localized using cut-off functions. Similar estimates for a regularized version of
\eqref{1.3} are employed in Section \ref{Local existence} to prove local existence of smooth 
solutions of \eqref{1.3} for smooth data \eqref{1.4}. 
Finally, in Section \ref{Asymptotics} the large-time behavior of smooth solutions to
\eqref{1.3} is discussed, finishing the proof of Theorem \ref{thm1.1}.  

\subsection{Notation}
The letter $C$ is used throughout to denote a generic constant, possibly depending on the 
``target'' $N$ and the initial energy $E(u_0)$. 

Moreover, since $T^{\perp}N$ by assumption is parallelizable and compact, there exists $\rho>0$ 
such that the representation 
\begin{equation*}
   T\colon N\times B_{\rho}(0;\R^m)\ni(p,y)\to p+\sum_{i=1}^my^i\nu_i(p)\in N_{\rho}
\end{equation*}
of the tubular neighborhood $N_{\rho}=\cup_{p\in N}B_{\rho}(p)$ of $N$ is a diffeomorphism, 
where $\nu_1,\dots,\nu_m$ is a suitable smooth orthonormal 
frame along $N$ and where we let $y=(y^1,\dots,y^m)\in\R^m$. 
For $q\in N_{\rho}$ then $T^{-1}(q)=(p,h)$ with $p=\pi_N(q)$ defines a (vector-valued) 
signed distance function $h=h(q)=(h^1(q),\dots,h^m(q))$ with $h^i(q)=\nu_i(p)\cdot(q-\pi_N(q))$
for each $1\le i\le i_0$. 
Fixing a smooth function $\eta\colon\R\to\R$ such that $\eta(s)=s$ for $|s|<\rho/2$, 
and with $\eta(s)=0$ for $|s|\ge 3\rho/4$, we then let
\begin{equation*}
    dist_N(q)=(dist^1_N(q),\dots,dist^m_N(q)),
\end{equation*}
with
 \begin{equation*}
   dist^i_N(q)=\eta(h^i(q))\hbox{ for }q\in N_{\rho},\ dist^i_N(q)=0\hbox{ else, }1\le i\le m.
\end{equation*}
Then for any smooth $u\in H^{1/2}(S^1;N)$ with harmonic extension $u\in H^1(B;\R^n)$ we have
\begin{equation}\label{1.7}
    \sum_{i=1}^m\nu_i(u)\partial_rdist^i_N(u)
    =\sum_{i=1}^m\nu_i(u)\nu_i(u)\cdot u_r=d\pi_N^{\perp}(u)u_r\ \hbox{ on } \partial B=S^1,
\end{equation}
where  for each $p\in N$ we denote as 
$d\pi_N^{\perp}(p)=1-d\pi_N(p)\colon\R^n\to T_p^{\perp}N$ the orthogonal projection.
In the sequel, we abbreviate
\begin{equation*}
  \sum_{i=1}^m\nu_i(u)\nu_i(u)\cdot u_r=:\nu(u)\nu(u)\cdot u_r=\nu(u)\partial_rdist_N(u);
 \end{equation*}
moreover, we extend the vector fields $\nu_i$ to the whole ambient space by letting
$\nu_i(q)=\nabla dist_N^i(q)$ for $q\in\R^n$, $1\le i\le m$.

Finally, we fix a smooth cut-off function $\varphi\in C_c^{\infty}(B)$ satisfying 
$0\le\varphi\le 1$ with $\varphi\equiv 1$ on $B_{1/2}(0)$, and for any $z_0\in B$, 
any $0<R<1$ we scale
\begin{equation*}
   \varphi_{z_0,R}(z)=\varphi((z-z_0)/R)\in C_c^{\infty}(B_R(z_0)).
\end{equation*}

\subsection*{Acknowledgement}
I thank Am\'elie Loher and the anonymous referee for careful reading of the
manuscript and useful suggestions.

\section{Energy inequality and first consequences}\label{Energy}
The half-harmonic heat flow may be regarded as the heat flow for the Dirichlet energy
in the class $H^{1/2}(S^1;N)$. Indeed, let $u(t)$ be a smooth solution 
of \eqref{1.3}, \eqref{1.4} for $0<t<T_0$. Then we have the following result.

\begin{lemma}\label{lemma2.1}
For any $0<T<T_0$ there holds 
\begin{equation*}
  E(u(T))+\int_0^T\int_{\partial B}|u_t|^2d\phi\;dt\le E(u_0).
\end{equation*}
\end{lemma}

\begin{proof} Integrating by parts and using \eqref{1.3} we compute
\begin{equation*}
 \begin{split}  
  \frac{d}{dt}E(u)=\int_B\nabla u&\nabla u_t\,dx
  =\int_{\partial B}u_r\cdot u_t\,d\phi\\
  &=-\int_{\partial B}|d\pi_N(u)u_r|^2d\phi=-\int_{\partial B}|u_t|^2d\phi
 \end{split}
\end{equation*}
for any $0<t<T_0$. The claim follows by integration.
\end{proof} 

Moreover, there holds a localized version of this energy inequality. 

\begin{lemma}\label{lemma2.2}
There exists a constant $C>0$ such that
for any $z_0\in B$, any $0<R<1$, any $\varepsilon>0$, and any 
$0<t_0<t_1\le t_0+\varepsilon R<T_0$ 
there holds 
\begin{equation*}
 \begin{split}
  \int_B|\nabla u(t_1)|^2\varphi_{z_0,R}^2dz
  +4\int_{t_0}^{t_1}\int_{\partial B}&|u_t|^2\varphi_{z_0,R}^2d\phi\,dt\\
  &\le4\int_B|\nabla u(t_0)|^2\varphi_{z_0,R}^2dz+C\varepsilon E(u_0).
 \end{split}
\end{equation*}
\end{lemma}

\begin{proof} 
Writing $\varphi=\varphi_{z_0,R}$ for brevity, integrating by parts,
and using Young's inequality, similar to the proof of Lemma \ref{lemma2.1}
for any $0<t<T_0$ we have
\begin{equation}\label{2.1}
 \begin{split}
  \frac{d}{dt}&\big(\frac12\int_B|\nabla u|^2\varphi^2dz\big)
  =\int_{\partial B}u_t\cdot u_r\varphi^2d\phi
  -\int_B u_tdiv(\nabla u\varphi^2)dz\\
  &=-\int_{\partial B}|d\pi_N(u)u_r|^2\varphi^2d\phi
  -2\int_B u_t\nabla u\varphi\nabla\varphi dz\\
  &\le-\int_{\partial B}|u_t|^2\varphi^2d\phi
  +(8\varepsilon R)^{-1}\int_B|\nabla u|^2\varphi^2dz
  +8\varepsilon R\int_B|u_t|^2|\nabla\varphi|^2dz.
 \end{split}
\end{equation}
Letting 
\begin{equation*}
  A=\sup_{t_0<t<t_1}\big(\frac12\int_B|\nabla u(t)|^2\varphi^2dz\big),
\end{equation*}
then upon integration we find
\begin{equation*}
 \begin{split}
  A&+\int_{t_0}^{t_1}\int_{\partial B}|u_t|^2\varphi^2d\phi\;dt\\
  &\le\int_B|\nabla u(t_0)|^2\varphi^2dz+\frac{t_1-t_0}{2\varepsilon R}A
  +C\varepsilon R^{-1}\int_{t_0}^{t_1}\int_{B_R(z_0)\cap B}|u_t|^2dz\,dt.
 \end{split}
\end{equation*}
But with $u=u(t)$ also $u_t=u_t(t)$ is harmonic for each $t$. Expanding 
\begin{equation*}
 u_t(re^{i\phi})=\sum_{k\ge 0}a_kr^ke^{ik\phi}
\end{equation*}
in a Fourier series, we see that the map 
\begin{equation*}
 r\mapsto\int_{\partial B_r(0)}|u_t|^2ds=2\pi\sum_{k\ge 0}|a_k|^2r^{2k+1},
\end{equation*}
with $ds$ denoting the element of length along $\partial B_r(0)$, is non-decreasing.
Thus for any $z_0\in B$, any $0<R<1$, and any $t_0<t<t_1$ there holds
\begin{equation}\label{2.2}
\begin{split}
 \int_{B_R(z_0)\cap B}&|u_t|^2dz\le 2R\int_{\partial B}|u_t|^2d\phi,
 \end{split}
\end{equation}
and we may use Lemma \ref{lemma2.1} to conclude.
\end{proof} 

\section{A regularity estimate}\label{A regularity estimate}
To illustrate the key ideas that later will allow us to prove higher regularity and 
analyze blow-up of solutions of \eqref{1.3}, 
we first consider smooth solutions $u\in H^{1/2}(S^1;N)$ of the equation 
\begin{equation}\label{3.1}
  d\pi_N(u)\partial_ru+f=0\ \hbox{ on }\partial B=S^1,
\end{equation}
where $f\in L^2(S^1)$. We prove the following a-priori estimate, where we use classical estimates
similar to Wettstein's \cite{Wettstein-2022} Lemma 3.4, which in turn is a fractional version of
an earlier result by Rivi\`ere \cite{Riviere-1993}. Note that with the truncated 
signed distance function $dist_N\colon\R^n\to\R^m$ we have the orthogonal decomposition
\begin{equation}\label{3.2}
  \partial_ru=d\pi_N(u)\partial_ru+d\pi_N^{\perp}(u)\partial_ru
  =d\pi_N(u)\partial_ru+\nu(u)\partial_r(dist_N(u))
\end{equation}
on $\partial B=S^1$, where we recall that we use the shorthand notation 
\begin{equation*}
   \nu(u)\partial_r(dist_N(u))=\sum_{i=1}^m\nu_i(u)\partial_r(dist^i_N(u))
   =\sum_{i=1}^m\nu_i(u)\nu_i(u)\cdot\partial_ru
\end{equation*}
and extend
$\nu_i(p)=\nabla dist^i_N(p)$, $p\in\R^n$.

\begin{proposition}\label{prop3.1}
There exist constants $C,\delta_0=\delta_0(N)>0$ such that for any smooth solution 
$u\in H^{1/2}(S^1;N)$ of \eqref{3.1} with $E(u)\le\delta^2<\delta_0^2$ there holds
\begin{equation}\label{3.3}
  \int_{S^1}|\partial_{\phi}u|^2d\phi\le C\|f\|^2_{L^2(S^1)}.
\end{equation}
\end{proposition}

\begin{proof}
Multiplying \eqref{3.2} with $\partial_ru$, we find the Pythagorean identity
\begin{equation}\label{3.4}
  |\partial_ru|^2=|d\pi_N(u)\partial_ru|^2+|d\pi^{\perp}_N(u)\partial_ru|^2
  =|d\pi_N(u)\partial_ru|^2+|\partial_r(dist_N(u))|^2.
\end{equation}

Note that $dist_N(u)\in H^1_0(B)$; moreover, for each $1\le i\le m$ we have 
$\nabla(dist^i_N(u))=\nu_i(u)\cdot\nabla u$, and there holds the equation 
\begin{equation}\label{3.5}
  \Delta(dist^i_N(u))=div(\nu_i(u)\cdot\nabla u)=\nabla u\cdot d\nu_i(u)\nabla u\ \hbox{ in }B.
\end{equation}
The divergence theorem now gives
\begin{equation*}
 \begin{split}
  \|\partial_r&(dist_N(u))\|^2_{L^2(S^1)}\\
  &=(\nabla(dist_N(u)),\nabla(dist_N(u))_r)_{L^2(B)}
  +(\Delta(dist_N(u)),(dist_N(u))_r)_{L^2(B)}\\
  &\le C\|\nabla u\|_{L^2(B)}\|\nabla^2(dist_N(u))\|_{L^2(B)}
  \le C\delta\|\nabla^2(dist_N(u))\|_{L^2(B)},
 \end{split}
\end{equation*}
where the basic $L^2$-theory for the Laplace equation \eqref{3.5} yields the bound
\begin{equation*}
  \|\nabla^2(dist_N(u))\|_{L^2(B)}\le C\|\Delta(dist^i_N(u))\|_{L^2(B)}
  \le C\|\nabla u\|^2_{L^4(B)}.
\end{equation*}
With Sobolev's embedding $H^{1/2}(B)\hookrightarrow L^4(B)$ we then conclude
\begin{equation*}
    \|\partial_r(dist_N(u))\|^2_{L^2(S^1)}\le C\delta\|\nabla u\|^2_{H^{1/2}(B)}.
\end{equation*}

Thus from \eqref{3.4} and \eqref{3.1} we have
\begin{equation}\label{3.6}
  \begin{split}
    \|\partial_ru\|^2_{L^2(S^1)}&\le\|f\|^2_{L^2(S^1)}+\|\partial_r(dist_N(u))\|^2_{L^2(S^1)}\\
    &\le\|f\|^2_{L^2(S^1)}+C\delta\|\nabla u\|^2_{H^{1/2}(B)}.
   \end{split}
\end{equation}
But Fourier expansion of the harmonic function $u$ gives
\begin{equation}\label{3.7}
  \|\partial_{\phi}u\|^2_{L^2(S^1)}=\|\partial_ru\|^2_{L^2(S^1)}=\frac12\|\nabla u\|^2_{L^2(S^1)}
\end{equation}
as well as the bound
\begin{equation*}
    \|\nabla u\|^2_{H^{1/2}(B)}\le C\|\nabla u\|^2_{L^2(S^1)},
\end{equation*}
and from \eqref{3.6} we obtain 
\begin{equation*}
  \begin{split}
    \|\partial_ru\|^2_{L^2(S^1)}&\le\|f\|^2_{L^2(S^1)}+C\delta\|\nabla u\|^2_{H^{1/2}(B)}
    \le\|f\|^2_{L^2(S^1)}+C\delta\|\partial_r u\|^2_{L^2(S^1)},
   \end{split}
\end{equation*}
which for sufficiently small $\delta>0$ by \eqref{3.7} yields the claim. 
\end{proof}

In particular, from Proposition \ref{prop3.1} we obtain a positive energy threshold for 
non-constant solutions of \eqref{1.6}.

\begin{corollary}\label{cor3.2}
Suppose $u\in H^{1/2}(S^1;N)$ smoothly solves \eqref{1.6}. Then, either $u$ is constant, or 
$E(u)\ge\delta_0^2$, with $\delta_0=\delta_0(N)>0$ given by Proposition \ref{prop3.1}. 
\end{corollary}

Combining the ideas in the proof of the previous result with ideas from the classical proof 
of the Courant-Lebesgue lemma in minimal surface theory, we can obtain the following
local version of Proposition \ref{prop3.1}. 

\begin{proposition}\label{prop3.3}
There exists a constant $\delta>0$ with the following property. Given any smooth solution 
$u\in H^{1/2}(S^1;N)$ of \eqref{3.1} with harmonic extension $u\in H^1(B)$, 
any $z_0\in\partial B$, and any $0<R\le 1/2$ such that 
\begin{equation}\label{3.8}
  \int_{B_R(z_0)\cap B}|\nabla u|^2dz<\delta^2,
\end{equation}
with a constant $C=C(R)>0$ there holds
\begin{equation*}
  \int_{B_{R^2}(z_0)\cap S^1}|\partial_{\phi}u|^2d\phi\le C\|f\|^2_{L^2(B_R(z_0)\cap S^1)}
  +CE(u).
\end{equation*}
\end{proposition}

\begin{proof} 
Fix any $z_0\in\partial B$ and $0<R\le 1/2$ such that \eqref{3.8} holds.
For suitable $\rho\in[R^2,R]$, with $s$ denoting arc-length along the curve 
$C_{\rho}=\{z_0+\rho e^{i\theta}\in B; \theta\in\R\}$
with end-points $z_j=z_0+\rho e^{i\theta_j}=e^{i\phi_j}\in\partial B$, $j=1,2$, we have 
\begin{equation*}
 \begin{split}
  \rho\int_{C_{\rho}}|\nabla u|^2ds
  \le 2\inf_{R^2<\rho'<R}\Big(\rho'\int_{C_{\rho'}}|\nabla u|^2ds\Big).
  \end{split}
\end{equation*}
We can bound the latter infimum by the average over $\rho\in[R^2,R]$ 
with respect to the measure with density $1/\rho$ to obtain the bound
\begin{equation}\label{3.9}
 \begin{split}
  \rho\int_{C_{\rho}}&|\nabla u|^2ds
  \le 2\int_{R^2}^R\int_{C_{\rho}}|\nabla u|^2ds\,d\rho\Big/\int_{R^2}^R\frac{d\rho}{\rho}\\
  &\le 2\int_B|\nabla u|^2dz\Big/|\log(R)|=4 E(u)/|\log(R)|.
 \end{split}
\end{equation}
Let $\Phi_0\colon B\to B$ be the conformal map fixing the circular arc
$C_{\rho}$ and mapping the point $z_0$ to the point $-z_0$, obtained 
as composition $\Phi_0=\pi_0^{-1}\circ\Psi_0\circ\pi_0$ of stereographic projection
$\pi_0\colon B\to\R^2_+$ from the point $-z_0$ and reflection $\Psi_0\colon\R^2_+\to\R^2_+$ 
of the upper half-plane $\R^2_+$ in the half-circle $\pi_0(C_{\rho})$.
Replacing $u$ by the map $u\circ\Phi_0$ in $B\setminus B_{\rho}(z_0)$
we obtain a piecewise smooth map $v_1\colon B\to\R^n$ which is harmonic on 
$B\setminus C_{\rho}$ and continuous on all of $B$. 
Let $v_0\in H^1(B)$ be harmonic with $w:=v_1-v_0\in H_0^1(B)$.
Note that by the variational characterization of harmonic functions and conformal 
invariance of the Dirichlet integral we have 
\begin{equation}\label{3.10}
  E(v_0)\le E(v_1)\le\int_{B_R(z_0)\cap B}|\nabla u|^2dz\le\delta^2.
\end{equation}
Moreover, for any smooth $\varphi\in H_0^1(B)$ by \eqref{3.9} we can estimate
\begin{equation*}
 \begin{split}
  \big|\int_B&\nabla w\nabla\varphi dz\big|=\big|\int_B\nabla v_1\nabla\varphi dz\big|
  =\big|\int_{C_{\rho}}[\partial_{\nu}v_1]\varphi ds\big|\\
  &\le\big(\int_{C_{\rho}}|\nabla u|^2ds\Big)^{1/2}
  \big(\int_{C_{\rho}}|\varphi|^2ds\Big)^{1/2}
  \le C(R)E(u)^{1/2}\|\varphi\|_{H^{1/2}(B)},
  \end{split}
\end{equation*}
where $[\partial_{\nu}v_1]$ denotes the difference of the outer and inner normal derivatives
of $v_1$ along $C_{\rho}$.
Thus we have $\Delta w\in H^{-1/2}(B)$, and the basic $L^2$-theory for the Laplace equation 
gives $w\in H^{3/2}\cap H_0^1(B)$ with 
\begin{equation*}
 \begin{split}
  \|w\|_{H^{3/2}(B)}
  \le\sup_{\varphi\in H_0^1(B),\|\varphi\|_{H^{1/2}(B)}\le 1}
  \big(\int_B\nabla w\nabla\varphi dz\big)\le C(R)E(u)^{1/2}
  \end{split}
\end{equation*}
and then also 
\begin{equation}\label{3.11}
  \|\partial_rw\|^2_{L^2(S^1)}\le C\|w\|^2_{H^{3/2}(B)}\le C(R)E(u).
\end{equation}

In view of \eqref{3.10}, for sufficiently small $\delta>0$
from Proposition \ref{prop3.1} we obtain the estimate 
\begin{equation}\label{3.12}
    \|\partial_{\phi}v_0\|^2_{L^2(S^1)}\le C\|d\pi_N(v_0)\partial_rv_0\|^2_{L^2(S^1)}.
\end{equation}
Observe that since $v_0=v_1$ on $\partial B=S^1$ and since we also have $v_1=u$ on 
$B\cap B_{\rho}(z_0)$, $v_1=u\circ\Phi_0$ on $B\setminus B_{\rho}(z_0)$, 
respectively, we can bound
\begin{equation*}
 \begin{split}
    \|d\pi_N(v_0)\partial_rv_0\|^2_{L^2(S^1)}&=\|d\pi_N(v_1)\partial_rv_0\|^2_{L^2(S^1)}\\
    &\le 2\|d\pi_N(v_1)\partial_rv_1\|^2_{L^2(S^1)}+2\|\partial_rw\|^2_{L^2(S^1)}
 \end{split}
\end{equation*}
and
\begin{equation*}
    \|d\pi_N(v_1)\partial_rv_1\|^2_{L^2(S^1)}
    \le C(R)\|d\pi_N(u)\partial_ru\|^2_{L^2(S^1\cap B_{\rho}(z_0))}.
\end{equation*}
Thus from \eqref{3.11} we obtain
\begin{equation*}
  \begin{split}
    \|d\pi_N(v_0)\partial_rv_0\|^2_{L^2(S^1)}
    &\le C(R)\|d\pi_N(u)\partial_ru\|^2_{L^2(S^1\cap B_{\rho}(z_0)}
    +C\|\partial_rw\|^2_{L^2(S^1)}\\
    &\le C(R)\|f\|^2_{L^2(S^1\cap B_{\rho}(z_0))}+C(R)E(u),
   \end{split}
\end{equation*}
and from \eqref{3.12} there results the bound
\begin{equation*}
  \begin{split}
  \|\partial_{\phi}&u\|^2_{L^2(S^1\cap B_{\rho}(z_0))}
  =\|\partial_{\phi}v_0\|^2_{L^2(S^1\cap B_{\rho}(z_0))}
  \le\|\partial_{\phi}v_0\|^2_{L^2(S^1)}\\
  &\le C\|d\pi_N(v_0)\partial_rv_0\|^2_{L^2(S^1)}
  \le C(R)\|f\|^2_{L^2(S^1\cap B_R(z_0)))}+C(R)E(u),
   \end{split}
\end{equation*}
as claimed.
\end{proof}

The local estimate Proposition \ref{prop3.3} also implies the following global bound.

\begin{proposition}\label{prop3.4}
There exists a constant $\delta>0$ with the following property. Given any smooth solution 
$u\in H^{1/2}(S^1;N)$ of \eqref{3.1}, any $0<R\le 1/2$ with 
\begin{equation}\label{3.13}
  \sup_{z_0\in B}\int_{B_R(z_0)\cap B}|\nabla u|^2dz<\delta^2,
\end{equation}
there holds
\begin{equation*}
  \int_{S^1}|\partial_{\phi}u|^2d\phi\le C(R)\|f\|^2_{L^2(S^1)}
  +C(R)E(u).
\end{equation*}
\end{proposition}

\begin{proof}
Covering $\partial B$ with balls $B_{R^2}(z_i)$, $1\le i\le i_0$, from 
Proposition \ref{prop3.3} we obtain the claim.
\end{proof}

\begin{remark}\label{rem3.5}
The proofs of the above propositions only require $u\in H^1(S^1;N)$ with harmonic extension
$u\in H^{3/2}(B)$.
\end{remark}

\section{Higher regularity}\label{Higher regularity}
Again let $u(t)$ be a smooth solution of the half-harmonic heat flow \eqref{1.3}
for $0<t<T_0$ with smooth initial data \eqref{1.4}. We show that as long as the flow 
does not concentrate energy in the sense of Theorem \ref{thm1.1}.ii) the solution remains
smooth and can be a-priori bounded in any $H^k$-norm in terms of the data. 

\subsection{$H^2$-bound}
In a first step we show an $L^2$-bound in space-time for the second derivatives of
our solution to the flow \eqref{1.3}. Recall that by harmonicity, 
writing $u=u(t)$, $\partial_{\phi}u=u_{\phi}$, and so on, for any $0<t<T_0$ we have
\eqref{3.7}, that is, 
\begin{equation*}
  \int_{\partial B}|u_{\phi}|^2d\phi=\int_{\partial B}|u_r|^2d\phi,
\end{equation*}
as Fourier expansion shows, with similar identities for partial derivatives of $u$ 
of higher order. Indeed, writing 
\begin{equation}\label{4.1}
  \Delta u=\frac{1}{r}(ru_r)_r+\frac{1}{r^2}u_{\phi\phi}
\end{equation}
we see that also $\partial_{\phi}^ju$ and then also $\nabla^{k-j}\partial_{\phi}^ju$ 
is harmonic for any $j\le k$ in $\N_0$, where $\nabla u=(u_x,u_y)$ in Euclidean 
coordinates $z=x+iy$. Thus by induction we obtain
\begin{equation}\label{4.2}
  \int_{\partial B}|\nabla^ku|^2d\phi=2\int_{\partial B}|\nabla^{k-1}u_{\phi}|^2d\phi
  =\dots=2^k\int_{\partial B}|\partial_{\phi}^ku|^2d\phi
\end{equation}
for any $k\in\N$. 
Similarly, for any $1/4<r<1$ with uniform constants $C>0$ we have  
\begin{equation*}
  \int_{\partial B_r(0)}|\nabla^ku|^2dz
  \le C\int_{\partial B_r(0)}|\nabla^{k-1}u_{\phi}|^2dz
  \le\dots\le C\int_{\partial B_r(0)}|\partial_{\phi}^ku|^2dz.
\end{equation*}
Integrating, and using the mean value property of harmonic functions together with \eqref{4.2}
to bound 
\begin{equation*}
  \sup_{B_{1/4}(0)}|\nabla^ku|^2
  \le C\int_{B\setminus B_{1/4}(0)}|\nabla^ku|^2dz\le C\int_B|\nabla\partial_{\phi}^{k-1}u|^2dz,
\end{equation*}
in particular, for any $k\in\N$ we have the bound
\begin{equation}\label{4.3}
  \int_B|\nabla^ku|^2dz\le C\int_B|\nabla\partial_{\phi}^{k-1}u|^2dz
\end{equation}
with an absolute constant $C>0$. 

The following lemma is strongly reminiscent of analogous results
for the harmonic map heat flow in two space dimensions. 

\begin{lemma}\label{lemma4.1}
With a constant $C>0$ depending only on $N$ there holds 
\begin{equation*}
  \frac{d}{dt}\big(\int_{\partial B}|u_{\phi}|^2d\phi\big)+\int_B|\nabla u_{\phi}|^2dz
  \le C\int_B|\nabla u|^2|u_{\phi}|^2dz.
\end{equation*}
\end{lemma}

\begin{proof} 
Writing $d\pi_N(u)=1-d\pi^{\perp}_N(u)$ with 
\begin{equation*}
     d\pi^{\perp}_N(u)X=\nu(u)\nu(u)\cdot X=\sum_{i=1}^m\nu_i(u)\nu_i(u)\cdot X
\end{equation*}
for any $X\in\R^n$, we compute
\begin{equation*}
 \begin{split}
  \frac12\frac{d}{dt}&\big(\int_{\partial B}|u_{\phi}|^2d\phi\big)
  =\int_{\partial B}u_{\phi}\cdot u_{\phi,t}d\phi
  =-\int_{\partial B}u_{\phi\phi}\cdot u_td\phi\\
  &=\int_{\partial B}u_{\phi\phi}\cdot d\pi_N(u)u_rd\phi
  =-\int_{\partial B}\big(u_{\phi}\cdot u_{r\phi}
  -u_{\phi}\cdot\partial_{\phi}(\nu(u)\,\nu(u)\cdot u_r)\big)d\phi\\
  &=-\frac12\int_{\partial B}\partial_r(|u_{\phi}|^2)d\phi 
  -\int_{\partial B}u_{\phi}\cdot d\nu(u)u_{\phi}\,\nu(u)\cdot u_rd\phi,
 \end{split}
\end{equation*}
where we use orthogonality $u_{\phi}\cdot\nu_i(u)=0$ on $\partial B$, $1\le i\le m$,
in the last step. 
But $u_{\phi}$ is harmonic. So with $\Delta|u_{\phi}|^2=2|\nabla u_{\phi}|^2$, 
from Gauss' theorem we obtain
\begin{equation*}
   \frac12\int_{\partial B}\partial_r(|u_{\phi}|^2)d\phi=\int_B|\nabla u_{\phi}|^2dz.
\end{equation*}
On the other hand, by Young's inequality we can estimate
\begin{equation*}
 \begin{split}
  \int_{\partial B}&u_r\cdot\nu(u)\,u_{\phi}\cdot d\nu(u)u_{\phi}d\phi
  =\int_B\nabla u\cdot\nabla\big(\nu(u)\,u_{\phi}\cdot d\nu(u)u_{\phi}\big)dz\\
  &\le C\int_B|\nabla u_{\phi}||\nabla u||u_{\phi}|dz
  +C\int_B|\nabla u|^2|u_{\phi}|^2dz\\
  &\le\frac12\int_B|\nabla u_{\phi}|^2dz
  +C\int_B|\nabla u|^2|u_{\phi}|^2dz,
 \end{split}
\end{equation*}
and our claim follows.
\end{proof}

Combining the previous result with a quantitative bound for the concentration of 
energy, we obtain a space-time bound for the second derivatives of $u$. 
Note that since $u$ is smooth by assumption,
for any $\delta>0$, any $T<T_0$ there exists a number $R=R(T,u)>0$ such that 
\begin{equation}\label{4.4}
   \sup_{z_0\in B,\,0<t<T}\int_{B_R(z_0)\cap B}|\nabla u(t)|^2dz<\delta.
\end{equation} 

\begin{proposition}\label{prop4.2}
There exist constants $\delta=\delta(N)>0$ and $C>0$ such that for any $T<T_0$ with 
$R>0$ as in \eqref{4.4} there holds 
\begin{equation}\label{4.5}
 \begin{split}
  \sup_{0<t<T}\int_{\partial B}|u_{\phi}(t)|^2d\phi
  &+\int_0^T\int_B|\nabla u_{\phi}|^2dx\,dt\\
  &\le C\int_{\partial B}|u_{0,\phi}|^2d\phi+CTR^{-2}E(u_0).
 \end{split}
\end{equation}
\end{proposition}

\begin{proof}
For given $T<T_0$ and $\delta>0$ to be determined we fix $R>0$ such 
that \eqref{4.4} holds. Let $B_{R/2}(z_i)$, $1\le i\le i_0$,
be a cover of $B$ such that any point $z_0\in B$ belongs to at most $L$ of the balls 
$B_R(z_i)$, where $L\in\N$ is independent of $R>0$. We then split
\begin{equation*}
   \int_B|\nabla u|^2|u_{\phi}|^2dz
   \le\sum_{i=1}^{i_0}\int_{B_{R/2}(z_i)}|\nabla u|^4dz
   \le\sum_{i=1}^{i_0}\int_B|\nabla(u\varphi_{z_i,R})|^4dz.
\end{equation*}
Using the multiplicative inequality \eqref{A.2} in the Appendix for each $i$ we can bound 
\begin{equation*}
 \int_B|\nabla(u\varphi_{z_i,R})|^4dz
 \le C\delta\int_{B_R(z_i)}\big(|\nabla^2u|^2+R^{-2}|\nabla u|^2\big)dz.
\end{equation*}
Summing over $1\le i\le i_0$, we thus obtain the bound 
\begin{equation*}
 \begin{split}
  \int_B|\nabla u|^2|u_{\phi}|^2dz
  &\le CL\delta\int_B|\nabla^2u|^2dz+CL\delta R^{-2}E(u)\\
  &\le CL\delta\int_B|\nabla u_{\phi}|^2dz+CL\delta R^{-2}E(u_0),
 \end{split}
\end{equation*}
and for sufficiently small $\delta>0$ from Lemma \ref{lemma4.1} we obtain the claim.
\end{proof}

With the help of Proposition \ref{prop4.2} we can now bound $u$ in $H^2(B)$ also 
uniformly in time. For this, we first note the following estimate, which also will 
be useful later for bounding higher order derivatives.

\begin{lemma}\label{lemma4.3}
For any $k\in\N$, with a constant $C>0$ depending only on $k$ and $N$,  
for the solution $u=u(t)$ to \eqref{1.3}, \eqref{1.4} for any $0<t<T_0$ there holds 
\begin{equation*}
 \begin{split}
  \frac{d}{dt}\big(\|\nabla\partial^k_{\phi}&u\|^2_{L^2(B)}\big)
  +\|\partial^k_{\phi}u_r\|^2_{L^2(S^1)}\\
  &\le C\sum_{1\le j_i\le k+1,\,\Sigma_ij_i\le k+2}\|\nabla\partial^k_{\phi}u\|_{L^2(B)}
  \|\Pi_i\nabla^{j_i}u\|_{L^2(B)}.
 \end{split}
\end{equation*}
\end{lemma}

\begin{proof}
For any $k\in\N$ we use harmonicity of $\partial^{2k}_{\phi}u$ to compute
\begin{equation}\label{4.6}
 \begin{split}
 \frac12&\frac{d}{dt}\big(\|\nabla\partial^k_{\phi}u\|^2_{L^2(B)}\big)
 =(-1)^k\int_B\nabla\partial^{2k}_{\phi}u\nabla u_t\;dx\\
 &=(-1)^k(\partial^{2k}_{\phi}u_r,u_t)_{L^2(S^1)}
 =(-1)^{k+1}(\partial^{2k}_{\phi}u_r,d\pi_N(u)u_r)_{L^2(S^1)}\\
 &=-(\partial^k_{\phi}u_r,\partial^k_{\phi}u_r)_{L^2(S^1)}
 +(\partial^k_{\phi}u_r,\partial^k_{\phi}(\nu(u)\,\nu(u)\cdot u_r))_{L^2(S^1)}\\
 &=-\|\partial^k_{\phi}u_r\|^2_{L^2(S^1)}+I,
 \end{split}
\end{equation}
where we split $I=\sum_{j=0}^k\Big({k\atop j}\Big)I_j$ with
\begin{equation*}
 \begin{split}
   I_j&=(\partial^k_{\phi}u_r,\partial^j_{\phi}(\nu(u)\,\nu(u))
   \partial^{k-j}_{\phi}u_r)_{L^2(S^1)}\\
   &=(\nabla\partial^k_{\phi}u,\nabla(\partial^j_{\phi}(\nu(u)\nu(u))
   \cdot\partial^{k-j}_{\phi}u_r))_{L^2(B)}.
 \end{split}
\end{equation*}
Hence for any $1\le j\le k$ we can bound
\begin{equation*}
 \begin{split}
 |I_j|&\le C\sum_{0\le i\le j}\|\nabla\partial^k_{\phi}u\|_{L^2(B)}
 \|\nabla\partial^{j-i}_{\phi}\nu(u)\partial^{i}_{\phi}\nu(u)
 \partial^{k-j}_{\phi}u_r\|_{L^2(B)}\\
 &+C\sum_{0\le i\le j}\|\nabla\partial^k_{\phi}u\|_{L^2(B)}
 \|\partial^{j-i}_{\phi}\nu(u)\partial^i_{\phi}\nu(u)
 \nabla\partial^{k-j}_{\phi}u_r\|_{L^2(B)}\\
 &\le C\sum_{1\le j_i\le k+1,\, \Sigma_ij_i=k+2}\|\nabla\partial^k_{\phi}u\|_{L^2(B)}
 \|\Pi_i\nabla^{j_i}u\|_{L^2(B)},
 \end{split}
\end{equation*}
as claimed. It remains to bound the term 
$I_0=\|\partial^k_{\phi}u_r\cdot\nu(u)\|^2_{L^2(S^1)}$.
With the signed distance function we can express
\begin{equation*}
  \nu(u)\cdot u_{\phi r}=\big(\nu(u)\cdot u_r\big)_{\phi}-u_r\cdot d\nu(u)u_{\phi}
  =(dist_N(u))_{\phi r}-u_r\cdot d\nu(u)u_{\phi},
\end{equation*}
so that 
\begin{equation*}
 \begin{split}
 I_0=\|\partial^k_{\phi}u_r&\cdot\nu(u)\|^2_{L^2(S^1)}
 =\big(\partial^k_{\phi}u_r\cdot\nu(u),\partial^k_{\phi}(dist_N(u))_r\big)_{L^2(S^1)}+II\\
 &=\big(\nabla\partial^k_{\phi}u,
 \nabla\big(\nu(u)\partial^k_{\phi}(dist_N(u))_r\big)\big)_{L^2(B)}+II,
 \end{split}
\end{equation*}
where all terms in $II$ can be dealt with as in the case $1\le j\le k$. Finally,
we have 
\begin{equation*}
 \begin{split}
 \big(\nabla&\partial^k_{\phi}u,
 \nabla\big(\nu(u)\partial^k_{\phi}(dist_N(u))_r\big)\big)_{L^2(B)}\\
 &\le\|\nabla\partial^k_{\phi}u\|_{L^2(B)}
 \big(\|\nabla^2\partial^k_{\phi}(dist_N(u))\|_{L^2(B)}
 +\|\nabla\nu(u)\partial^k_{\phi}(dist_N(u))_r\|_{L^2(B)}\big).
\end{split}
\end{equation*}
But by the chain rule we can bound
\begin{equation*}
 \begin{split}
  \|\nabla\nu(u)\partial^k_{\phi}&(dist_N(u))_r\|_{L^2(B)}\big)
  \le C\|\nabla u\nabla^{k+1}(dist_N(u))\|_{L^2(B)}\big)\\
  &\le C\sum_{1\le j_i\le k+1,\,\Sigma_ij_i=k+2}\|\Pi_i\nabla^{j_i}u\|_{L^2(B)}.
\end{split}
\end{equation*}
Moreover, by \eqref{3.5} and elliptic regularity theory, there holds
\begin{equation*}
 \begin{split}
  \|\nabla^{k+2}&(dist_N(u))\|^2_{L^2(B)}\le C\|\Delta(dist_N(u))\|^2_{H^k(B)}
  \le C\|\nabla u\cdot d\nu_i(u)\nabla u\|^2_{H^k((B)}\\
  &\le C\sum_{1\le j_i\le k+1,\,\Sigma_ij_i\le k+2}\|\Pi_i\nabla^{j_i}u\|_{L^2(B)},
 \end{split}
\end{equation*}
which gives the claim.
\end{proof}

For $k=1$, from Proposition \ref{prop4.2} we now easily derive a uniform $L^2$-bound for
the second derivatives of the flow.

\begin{proposition}\label{prop4.4}
For any smooth $u_0\in H^{1/2}(S^1;N)$ and any $T<T_0$ with $R>0$ as in 
Proposition \ref{prop4.2} with a constant $C_1=C_1(T,R,u_0)>0$ depending on the 
right hand side of \eqref{4.5} there holds 
\begin{equation*}
 \begin{split}
  \sup_{0<t<T}&\int_B|\nabla u_{\phi}(t)|^2dz
  +\int_0^{T}\int_{\partial B}|u_{\phi r}|^2d\phi\,dt
  \le C_1\int_B|\nabla u_{0,\phi}|^2dz+C_1.
 \end{split}
\end{equation*}
\end{proposition}

\begin{proof}
For $k=1$ by Lemma \ref{lemma4.3} we need to bound the term
\begin{equation*}
 \begin{split}
 J=\sum_{1\le j_i\le 2,\, \Sigma_ij_i\le 3}\|\Pi_i\nabla^{j_i}u\|_{L^2(B)}
 \le C\||\nabla^2 u||\nabla u|+|\nabla u|^3\|_{L^2(B)}+J_1,
 \end{split}
\end{equation*}
where $J_1$ contains all terms of lower order.
By the maximum principle and Sobolev's embedding 
$H^1(\partial B)\hookrightarrow L^{\infty}(\partial B)$ we can estimate
\begin{equation*}
 \begin{split}
 \|\nabla u\|^2_{L^{\infty}(B)}\le\|\nabla u\|^2_{L^{\infty}(\partial B)}
 \le C\|\nabla u\|^2_{H^1(\partial B)}
 \le C\|u_{\phi r}\|^2_{L^2(\partial B)}+C_1, 
 \end{split}
\end{equation*}
where we have also used \eqref{3.7} and Proposition \ref{prop4.2}. Also bounding 
\begin{equation*}
 \begin{split}
 \|\nabla u\|^3_{L^6(B)}&\le\|\nabla u\|^2_{L^4(B)}\|\nabla u\|_{L^{\infty}(B)}\\
 &\le C\big(\|\nabla^2 u\|_{L^2(B)}\|\nabla u\|_{L^2(B)}+E(u)\big)\|\nabla u\|_{L^{\infty}(B)}
 \end{split}
\end{equation*}
via \eqref{A.2}, and again using \eqref{3.7} 
(and with similar, but simpler bounds for $J_1$), we arrive at the estimate
\begin{equation*}
 \begin{split}
 J&\le C\||\nabla^2 u||\nabla u|+|\nabla u|^3\|_{L^2(B)}+C_1\\
 &\le C\big(\|\nabla^2 u\|_{L^2(B)}+E(u)\big)\|\nabla u\|_{L^{\infty}(B)}+C_1\\
 &\le C\big(1+\|\nabla u_{\phi}\|_{L^2(B)}+E(u_0)\big)
 \big(\|u_{\phi r}\|_{L^2(\partial B)}+C_1\big). 
 \end{split}
\end{equation*}
With Lemma \ref{lemma4.3} and Young's inequality we then have
\begin{equation}\label{4.7}
 \begin{split}
  \frac{d}{dt}&\big(1+\|\nabla u_{\phi}\|^2_{L^2(B)}\big)+\|u_{\phi r}\|^2_{L^2(S^1)}\\
  &\le C\|\nabla u_{\phi}\|_{L^2(B)}\big(\|\nabla u_{\phi}\|_{L^2(B)}+E(u_0)\big)
 \big(\|u_{\phi r}\|_{L^2(\partial B)}+C_1\big)\\
  &\le\frac12\|u_{\phi r}\|^2_{L^2(\partial B)}
  +C(1+\|\nabla u_{\phi}\|^2_{L^2(B)})\big(\|\nabla u_{\phi}\|^2_{L^2(B)}+C_1\big).
 \end{split}
\end{equation}
Absorbing the first term on the right on the left hand side of this inequality and
dividing by $1+\|\nabla u_{\phi}\|^2_{L^2(B)}$ we obtain 
\begin{equation*}
 \begin{split}
  \frac{d}{dt}&\big(\log\big(1+\|\nabla u_{\phi}\|^2_{L^2(B)}\big)\big)
  \le C\|\nabla u_{\phi}\|^2_{L^2(B)}+C_1,
 \end{split}
\end{equation*}
and from Proposition \ref{prop4.2} we obtain the bound 
\begin{equation*}
 \begin{split}
   \sup_{0<t<T}&\|\nabla u_{\phi}(t)\|^2_{L^2(B)}\le C_1(1+\|\nabla u_{0,\phi}\|^2_{L^2(B)}).
 \end{split}
\end{equation*}
The claim then follows from \eqref{4.7}.
\end{proof}

\subsection{$H^3$-bounds}
The derivation of a-priori $L^2$-bounds for third derivatives of
the solution $u$ to the flow \eqref{1.3}, \eqref{1.4} requires special care, which 
is why we highlight this case. 

\begin{proposition}\label{prop4.5}
For any smooth $u_0\in H^{1/2}(S^1;N)$ and any $T<T_0$ there holds 
\begin{equation*}
  \sup_{0<t<T}\int_B|\nabla u_{\phi\phi}(t)|^2dz
  +\int_0^T\int_{\partial B}|u_{\phi\phi r}|^2d\phi\,dt
  \le C_2\int_B|\nabla u_{0,\phi\phi}|^2dz+C_2,
\end{equation*}
where we denote as $C_2=C_2(T,R,u_0)>0$ a constant bounded by the terms on the right hand 
side in the statements of Propositions \ref{prop4.2} and \ref{prop4.4}.
\end{proposition}

\begin{proof}
For $k=2$ by Lemma \ref{lemma4.3} we need to bound the term
\begin{equation*}
 \begin{split}
 J&=\sum_{1\le j_i\le 3,\, \Sigma_ij_i=4}
 \|\Pi_i\nabla^{j_i}u\|_{L^2(B)}\\
 &\le C\||\nabla u|^4+|\nabla u|^2|\nabla^2 u|+|\nabla^2 u|^2+|\nabla u||\nabla^3 u|\|_{L^2(B)}
 \end{split}
\end{equation*}
and corresponding terms involving at most $3$ derivatives in total, which we will omit.

In dealing with the first term, 
by the multiplicative inequality \eqref{A.2} and Sobolev's embedding 
$H^2(B)\hookrightarrow L^{\infty}(B)$ we can estimate
\begin{equation*}
 \begin{split}
 \|\nabla u\|^4_{L^8(B)}&\le\|\nabla u\|^2_{L^4(B)}\|\nabla u\|^2_{L^{\infty}(B)}
 \le C\|\nabla u\|_{H^1(B)}\|\nabla u\|_{L^2(B)}\|\nabla u\|^2_{L^{\infty}(B)}\\
 &\le C(\|\nabla^2 u\|^2_{L^2(B)}+E(u))\|\nabla u\|^2_{L^{\infty}(B)}
 \le C_2\|\nabla u\|^2_{L^{\infty}(B)}\\
 &\le C_2(\|\nabla^3 u\|_{L^2(B)}+\|\nabla u\|_{L^2(B)})\|\nabla u\|_{L^{\infty}(B)}
 \end{split}
\end{equation*}
with a constant $C_2=C_2(T,R,u_0)>0$ as in the statement of the proposition.
Similarly there holds 
\begin{equation*}
 \begin{split}
 \|\nabla^2& u\|^2_{L^4(B)}\le C\|\nabla^2 u\|_{H^1(B)}\|\nabla^2 u\|_{L^2(B)}\\
 &\le\|\nabla^3 u\|_{L^2(B)}\|\nabla^2 u\|_{L^2(B)}+\|\nabla^2 u\|^2_{L^2(B)}
 \le C_2(1+\|\nabla^3 u\|_{L^2(B)}).
 \end{split}
\end{equation*}
Hence we can also bound
\begin{equation*}
 \begin{split}
 \||\nabla u|^2&|\nabla^2 u|\|_{L^2(B)}\le\|\nabla u\|^4_{L^8(B)}+\|\nabla^2 u\|^2_{L^4(B)}\\
 &\le C_2(1+\|\nabla^3 u\|_{L^2(B)})(1+\|\nabla u\|_{L^{\infty}(B)}).
 \end{split}
\end{equation*}
Finally, we estimate
\begin{equation*}
 \begin{split}
 \||\nabla u||\nabla^3 u|\|_{L^2(B)}
 \le\|\nabla^3 u\|_{L^2(B)}\|\nabla u\|_{L^{\infty}(B)}
 \end{split}
\end{equation*}
to obtain 
\begin{equation*}
 J\le C_2(1+\|\nabla^3 u\|_{L^2(B)})(1+\|\nabla u\|_{L^{\infty}(B)}).
\end{equation*}

But with the inequality
\begin{equation*}
 \|f\|_{L^{\infty}(B)}\le C\|f\|_{H^1(B)}(1+\log^{1/2}(1+\|f\|_{H^2(B)}/\|f\|_{H^1(B)})
\end{equation*}
for $f\in H^2(B)$ due to Brezis-Gallouet \cite{Brezis-Gallouet-1980}
(see also Brezis-Wainger \cite{Brezis-Wainger-1980} for a more general version) 
we have
\begin{equation*}
 \begin{split}
 &\|\nabla u\|^2_{L^{\infty}(B)}
 \le C\|\nabla u\|^2_{H^1(B)}\big(1+\log(1+\|\nabla u\|_{H^2(B)}/\|\nabla u\|_{H^1(B)})\big)\\
 &\quad\le C_2(1+\log(1+\|\nabla^3 u\|_{L^2(B)})),
 \end{split}
\end{equation*}
and Lemma \ref{lemma4.3} yields the differential inequality 
\begin{equation*}
 \begin{split}
  \frac{d}{dt}&\big(\|\nabla\partial^2_{\phi}u\|^2_{L^2(B)}\big)
  +\|u_{\phi\phi r}\|^2_{L^2(\partial B)}\\
  &\le C_2\|\nabla\partial^2_{\phi}u\|_{L^2(B)}
  (1+\|\nabla^3 u\|_{L^2(B)})\big(1+\log(1+\|\nabla^3 u\|_{L^2(B)})\big).
 \end{split}
\end{equation*}
Simplifying, and recalling that 
$\|\nabla^3 u\|^2_{L^2(B)}\le C\|\nabla\partial^2_{\phi}u\|^2_{L^2(B)}$ by \eqref{4.3}, 
we then find
\begin{equation*}
 \begin{split}
  \frac{d}{dt}\big(1+&\|\nabla\partial^2_{\phi}u\|_{L^2(B)}\big)\\
  &\le C_2(1+\|\nabla\partial^2_{\phi}u\|_{L^2(B)})\big(1+\log(1+\|\nabla\partial^2_{\phi}u\|_{L^2(B)})\big);
 \end{split}
\end{equation*}
that is, we have
\begin{equation*}
 \begin{split}
  \frac{d}{dt}\big(1+&\log(1+\|\nabla\partial^2_{\phi}u\|_{L^2(B)})\big)
  \le C_2\big(1+\log(1+\|\nabla\partial^2_{\phi}u\|_{L^2(B)})\big).
 \end{split}
\end{equation*}
Arguing as in the proof of Proposition \ref{prop4.4} we then obtain the claim.
\end{proof}

\subsection{$H^m$-bounds, $m\ge 4$}
In view of Proposition \ref{prop4.5} we can now use induction to prove the following result. 

\begin{proposition}\label{prop4.6}
For any $k\ge 3$, any smooth $u_0\in H^{1/2}(S^1;N)$, and any $T<T_0$ there holds 
\begin{equation*}
  \sup_{0<t<T}\int_B|\nabla\partial^k_{\phi}(t)|^2dz
  +\int_0^T\int_{\partial B}|\partial^k_{\phi}u_r|^2d\phi\,dt
  \le C_k\int_B|\nabla\partial^k_{\phi}u_0|^2dz+C_k,
\end{equation*}
where we denote as $C_k=C_k(T,R,u_0)>0$ a constant bounded by the terms on the right hand 
side in the statement of the proposition for $k-1$.
\end{proposition}

\begin{proof}
By Proposition \ref{prop4.5} the claimed result holds true for $k=2$. 
Suppose the claim holds true for some $k_0\ge 2$ and let $k=k_0+1$. 
Note that by Sobolev's embedding $H^2(B)\hookrightarrow W^{1,4}\cap C^0(\bar{B})$ 
and \eqref{4.3} for $0\le t<T$ we then have the uniform bounds 
\begin{equation}\label{4.8}
 \begin{split}
  \|\nabla^{k_0+1}u\|^2_{L^2(B)}&+\|\nabla^{k_0}u\|^2_{L^4(B)}
  +\sum_{1\le j\le k_0-1}\|\nabla^{j}u\|^2_{L^{\infty}(B)}\\
  &\le C_{k_0}\|\nabla^{k_0+1}u_0\|^2_{L^2(B)}+C_{k_0}\le C_k<\infty
 \end{split}
\end{equation}
with a constant of the type $C_k$, as defined above.

By Lemma \ref{lemma4.3} again we only need to bound the term
\begin{equation*}
  J=\sum_{1\le j_i\le k+1,\,\Sigma_ij_i\le k+2}\|\Pi_i\nabla^{j_i}u\|_{L^2(B)}.
\end{equation*}
Clearly we have
\begin{equation*}
 \begin{split}
  J&\le\|\nabla^{k+1}u\|_{L^2(B)}\|\nabla u\|_{L^{\infty}(B)}
  +\|\nabla^ku\|_{L^2(B)}\|\nabla u\|^2_{L^{\infty}(B)}+\|\nabla^ku\nabla^2u\|_{L^2(B)}\\
  &\qquad+\|\nabla^{k-1}u\nabla^3u\|_{L^2(B)}
  +\|\nabla^{k-1}u\nabla^2u\|_{L^2(B)}\|\nabla u\|_{L^{\infty}(B)}+C_k\\
  &\le C_k\|\nabla^{k+1}u\|_{L^2(B)}+\|\nabla^ku\nabla^2u\|_{L^2(B)}
  +\|\nabla^{k-1}u\nabla^3u\|_{L^2(B)}+C_k.
 \end{split}
\end{equation*}

We now distinguish the following cases: If $k-1=k_0\ge 3$ by \eqref{4.8} we can bound 
\begin{equation*}
  \|\nabla^ku\nabla^2u\|_{L^2(B)}
  \le\|\nabla^ku\|_{L^2(B)}\|\nabla^2u\|_{L^{\infty}(B)}
  \le C_{k_0}\|\nabla^{k_0+1}u\|^2_{L^2(B)}+C_{k_0}\le C_k
\end{equation*}
as well as
\begin{equation*}
 \begin{split}
  \|\nabla^{k-1}u\nabla^3u\|_{L^2(B)}\le\|\nabla^{k-1}u\|_{L^4(B)}\|\nabla^3u\|_{L^4(B)}
  \le C_{k_0}\|\nabla^{k_0}u\|^2_{L^4(B)}+C_{k_0}\le C_k
 \end{split}
\end{equation*}
to obtain the estimate
\begin{equation*}
  J\le C_k\|\nabla^{k+1}u\|_{L^2(B)}+C_k.
\end{equation*}

If, on the other hand, $k_0=k-1=2$, by 
our induction hypothesis \eqref{4.8} we have
\begin{equation*}
 \begin{split}
  \|\nabla^{k-1}u&\nabla^3u\|_{L^2(B)}=\|\nabla^2u\nabla^ku\|^2_{L^2(B)}
  \le\|\nabla^ku\|_{L^4(B)}\|\nabla^2u\|_{L^4(B)}\\
  &\le C_k\|\nabla^ku\|_{H^1(B)}\le C_k\|\nabla^{k+1}u\|_{L^2(B)}+C_k,
 \end{split}
\end{equation*}
and we find
\begin{equation*}
  J\le C_k\|\nabla^{k+1}u\|_{L^2(B)}+C_k 
\end{equation*}
as before. 

In any case, inequality \eqref{4.3} and Lemma \ref{lemma4.3} now may be invoked to obtain
\begin{equation*}
 \begin{split}
  \frac{d}{dt}\big(\|\nabla\partial^k_{\phi}u\|^2_{L^2(B)}\big)
  \le C_k\|\nabla\partial^k_{\phi}u\|^2_{L^2(B)}+C_k,
 \end{split}
\end{equation*}
and our claim follows.
\end{proof}

\subsection{Local $H^2$-bounds}
The bounds established so far all require the initial data to be sufficiently smooth 
for the estimate at hand and do not yet allow to show smoothing of the flow.
For the latter purpose we next prove a second set of ``intermediate'' estimates that
in combination with the first set of estimates later will allow boot-strapping. 
Moreover, in contrast to the estimates established so far, the following estimates may 
be localized. This will be important for showing regularity of the flow at blow-up 
times away from concentration points of the energy on $\partial B$.

For the localized estimates, 
fix a point $z_0\in\partial B$ and some radius $0<R_0<1/4$ and for $k\in\N$ set 
$R_k=2^{-k}R_0$, $\varphi_k=\varphi_{z_0,R_k}$. Set $\varphi_k=1$ for each $k\in\N$
for the analogous global bounds.

We first establish the following localized version of Lemma \ref{lemma4.1}.

\begin{lemma}\label{lemma4.7}
With a constant $C>0$ depending only on $N$ there holds 
\begin{equation*}
  \frac{d}{dt}\big(\int_{\partial B}|u_{\phi}|^2\varphi_1^2\,d\phi\big)
  +\int_B|\nabla u_{\phi}|^2\varphi_1^2\,dz
  \le C\int_B|\nabla u|^2|u_{\phi}|^2\varphi_1^2\,dz+CR_0^{-2}E(u_0).
\end{equation*}
\end{lemma}

\begin{proof} 
Similar to the proof of Lemma \ref{lemma4.1}, we compute
\begin{equation*}
 \begin{split}
  &\frac12\frac{d}{dt}\big(\int_{\partial B}|u_{\phi}|^2\varphi_1^2\,d\phi\big)
  =\int_{\partial B}u_{\phi}\cdot u_{\phi,t}\varphi_1^2\,d\phi
  =-\int_{\partial B}\partial_{\phi}(u_{\phi}\varphi_1^2)\cdot u_td\phi\\
  &=\int_{\partial B}\partial_{\phi}(u_{\phi}\varphi_1^2)\cdot d\pi_N(u)u_rd\phi
  =-\int_{\partial B}\big(u_{\phi}\cdot u_{r\phi}
  -u_{\phi}\cdot\partial_{\phi}(\nu(u)\,\nu(u)\cdot u_r)\big)\varphi_1^2d\phi\\
  &=-\frac12\int_{\partial B}\partial_r(|u_{\phi}|^2)\varphi_1^2d\phi 
  -\int_{\partial B}u_{\phi}\cdot d\nu(u)u_{\phi}\,\nu(u)\cdot u_r\varphi_1^2d\phi.
 \end{split}
\end{equation*}
With $\Delta|u_{\phi}|^2=2|\nabla u_{\phi}|^2$ we obtain
\begin{equation*}
   \frac12\int_{\partial B}\partial_r(|u_{\phi}|^2)\varphi_1^2d\phi
   =\int_B|\nabla u_{\phi}|^2\varphi_1^2dz+\int_B\nabla |u_{\phi}|^2\varphi_1\nabla\varphi_1dz,
\end{equation*}
where 
\begin{equation*}
   \big|\int_B\nabla |u_{\phi}|^2\varphi_1\nabla\varphi_1dz\big|
   \le\frac14\int_B|\nabla u_{\phi}|^2\varphi_1^2dz+C\int_B |u_{\phi}|^2|\nabla\varphi_1|^2dz
\end{equation*}
by Young's inequality. Finally, we can bound
\begin{equation*}
 \begin{split}
  \int_{\partial B}&u_r\cdot\nu(u)\,u_{\phi}\cdot d\nu(u)u_{\phi}\varphi_1^2\,d\phi
  =\int_B\nabla u\cdot\nabla\big(\nu(u)\,u_{\phi}\cdot d\nu(u)u_{\phi}\varphi_1^2\big)dz\\
  &\le C\int_B\big(|\nabla u_{\phi}||\nabla u||u_{\phi}|+|\nabla u|^2|u_{\phi}|^2\big)\varphi_1^2dz
  +C\int_B|\nabla u||\nabla\varphi_1||u_{\phi}|^2\varphi_1dz\\
  &\le\frac14\int_B|\nabla u_{\phi}|^2\varphi_1^2dz
  +C\int_B|\nabla u|^2|u_{\phi}|^2\varphi_1^2dz+C\int_B|\nabla u|^2|\nabla\varphi_1|^2dz,
 \end{split}
\end{equation*}
and our claim follows.
\end{proof}

We need a substitute for the global bound \eqref{4.3}. 
For this, we note that the equation \eqref{4.1} 
also implies the pointwise bound $|u_{rr}|^2\le2|u_{\phi\phi}|^2/r^4+2|u_r|^2/r^2$; 
hence we have
\begin{equation*}
 \begin{split}
  |\nabla^2u|^2\le C(|\nabla u_{\phi}|^2+2|\nabla u|^2)\ \hbox{ in } B_{R_0}(z_0)
 \end{split}
\end{equation*}
with an absolute constant $C>0$, uniformly in $z_0\in\partial B$ and $0<R_0<1/4$.
By induction then, similarly we have 
\begin{equation}\label{4.9}
 \begin{split}
  |\nabla^{k+1}u|^2\le C(|\nabla^k\partial_{\phi}u|^2+|\nabla^ku|^2)
  \le C\sum_{j=0}^k|\nabla\partial_{\phi}^ju|^2\ \hbox{ in } B_{R_0}(z_0)
 \end{split}
\end{equation}
with an absolute constant $C=C(k)>0$, uniformly in $z_0\in\partial B$ and $0<R_0<1/4$ 
for any $k\in\N$. 

Likewise, as a substitute for the global non-concentration condition \eqref{4.4} 
we now suppose that $z_0\in\partial B$ is not a concentration point in the sense that 
for suitably chosen
$\delta>0$ to be determined in the sequel and some $0<R_0<1/4$ as above there holds
\begin{equation}\label{4.10}
   \sup_{0<t<T_0}\int_{B_{R_0}(z_0)\cap B}|\nabla u(t)|^2dz<\delta.
\end{equation} 
We then obtain the following localized version of Proposition \ref{prop4.2}.

\begin{proposition}\label{prop4.8}
There exist constants $\delta>0$ and $C>0$ independent of $R_0>0$ 
such that whenever \eqref{4.10} holds then for any $T\le T_0$ we have
\begin{equation*}
 \begin{split}
  \sup_{0<t<T}\int_{\partial B}|u_{\phi}(t)|^2\varphi_1^2\,d\phi
  &+\int_0^T\int_B|\nabla u_{\phi}|^2\varphi_1^2\,dz\,dt\\
  &\le2\int_{\partial B}|u_{0,\phi}|^2\varphi_1^2\,d\phi+CTR_0^{-2}E(u_0).
 \end{split}
\end{equation*}
\end{proposition}

\begin{proof}
With the help of the inequality \eqref{A.1} in the Appendix we can bound 
\begin{equation*}
 \begin{split}
  \int_B|\nabla u|^4&\varphi_1^2dz
 \le C\delta\int_{B_R(z_i)}|\nabla^2u|^2\varphi_1^2dz
 +C\delta R_0^{-2}\int_{B_R(z_i)}|\nabla u|^2dz.
 \end{split}
\end{equation*}
Thus, for sufficiently small $\delta>0$ our claim follows from Lemma \ref{lemma4.7}.
\end{proof}

The next lemma again prepares for a proposition that later will allow us to obtain 
higher derivative bounds by induction. Note the differences to Lemma \ref{lemma4.3}.

\begin{lemma}\label{lemma4.9}
For any $k\ge 2$, with a constant $C>0$ depending only on $k$ and $N$,  
for the solution $u=u(t)$ to \eqref{1.3}, \eqref{1.4} for any $0<t<T_0$ there holds 
\begin{equation*}
 \begin{split}
  \frac{d}{dt}&\big(\|\partial^k_{\phi}u\varphi_k\|^2_{L^2(\partial B)}\big)
  +\|\nabla\partial^k_{\phi}u\varphi_k\|^2_{L^2(B)}\\
  &\le C\sum_{1\le j_i\le k,\, \Sigma_ij_i\le 2k+2} \|\Pi_i\nabla^{j_i}u\varphi_k^2\|_{L^1(B)}\\
  &\quad+C\sum_{1\le j_i\le k,\,\Sigma_{i\ge 0}j_i\le k+1}
  \|\Pi_{i>0}\nabla^{j_i}u\nabla^{j_0}\varphi_k\|^2_{L^2(B)}+CR_0^{-2k}E(u_0).
 \end{split}
\end{equation*}
\end{lemma}

\begin{proof}
Fix $k\ge 2$. With $\Delta|\partial^k_{\phi}u|^2=2|\nabla\partial^k_{\phi}u|^2$ we compute
\begin{equation*}
 \begin{split}
 \frac12&\frac{d}{dt}\big(\|\partial^k_{\phi}u\varphi_k\|^2_{L^2(\partial B)}\big)
 =(-1)^k\int_{\partial B}\partial^k_{\phi}(\partial^k_{\phi}u\varphi_k^2)\cdot u_td\phi\\
 &=(-1)^{k+1}\int_{\partial B}\partial^k_{\phi}(\partial^k_{\phi}u\varphi_k^2)
 \cdot(u_r-\nu(u)\,\nu(u)\cdot u_r)d\phi\\
 &=-\frac12\int_{\partial B}\partial_r(|\partial^k_{\phi}u|^2)\varphi_k^2d\phi 
 +\int_{\partial B}\partial^k_{\phi}u\cdot\partial^k_{\phi}\big(\nu(u)\,\nu(u)\cdot u_r)\big)
  \varphi_k^2d\phi.\\
 &=-\int_B|\nabla\partial^k_{\phi}u|^2\varphi_k^2dz
 -\int_B\nabla(|\partial^k_{\phi}u|^2)\varphi_k\nabla\varphi_kdz+I,
 \end{split}
\end{equation*}
where we split 
\begin{equation*}
   I=\int_{\partial B}\partial^k_{\phi}u\cdot\partial^k_{\phi}
   \big(\nu(u)\,\nu(u)\cdot u_r)\big)\varphi_k^2d\phi=\sum_{j=0}^k\Big({k\atop j}\Big)I_j
\end{equation*}
with
\begin{equation*}
 \begin{split}
   I_j=(\partial^k_{\phi}u\cdot&\partial^j_{\phi}(\nu(u)\,\nu(u))\varphi_k^2,
   \partial^{k-j}_{\phi}u_r)_{L^2(\partial B)}\\
   &=\big(\nabla\big(\partial^k_{\phi}u\cdot\partial^j_{\phi}(\nu(u)\nu(u))\varphi_k^2\big),
   \nabla\partial^{k-j}_{\phi}u\big)_{L^2(B)},\ 0\le j\le k.
 \end{split}
\end{equation*}
For $1\le j\le k$ we bound
\begin{equation*}
 \begin{split}
 |I_j|&\le C\sum_{0\le i\le j}\|\nabla\partial^k_{\phi}u\varphi_k\|_{L^2(B)}
 \|\partial^{j-i}_{\phi}\nu(u)\partial^{i}_{\phi}\nu(u)
 \nabla\partial^{k-j}_{\phi}u\varphi_k\|_{L^2(B)}\\
 &\quad+C\sum_{0\le i\le j}\|\partial^k_{\phi}u\cdot
 \nabla\big(\partial^{j-i}_{\phi}\nu(u)\partial^{i}_{\phi}\nu(u)\varphi_k^2\big)\cdot
 \nabla\partial^{k-j}_{\phi}u\|_{L^1(B)}\\
  \end{split}
\end{equation*}
By the chain rule then for $1\le j\le k$ we have
\begin{equation*}
 \begin{split}
 |I_j|&\le C\sum_{1\le j_i\le k,\, \Sigma_ij_i=k+1}\|\nabla\partial^k_{\phi}u\varphi_k\|_{L^2(B)}
 \|\Pi_i\nabla^{j_i}u\varphi_k\|_{L^2(B)}\\
 &\quad+C\sum_{1\le j_i\le k,\, \Sigma_ij_i=k+2}
 \|\partial^k_{\phi}u\cdot\Pi_i\nabla^{j_i}u\varphi_k^2\|_{L^1(B)}\\
 &\quad+C\sum_{1\le j_i\le k,\, \Sigma_ij_i=k+1}
 \|\partial^k_{\phi}u\cdot\Pi_i\nabla^{j_i}u\varphi_k\nabla\varphi_k\|_{L^1(B)}.
 \end{split}
\end{equation*}
By Cauchy-Schwarz and Young's inequality then we can bound
\begin{equation*}
 \begin{split}
 &\sum_{1\le j\le k}|I_j|\le\frac14\|\nabla\partial^k_{\phi}u\varphi_k\|^2_{L^2(B)}
 + C\sum_{1\le j_i\le k,\, \Sigma_ij_i=k+1} \|\Pi_i\nabla^{j_i}u\varphi_k\|^2_{L^2(B)}\\
 &\qquad+C\sum_{1\le j_i\le k,\,\Sigma_ij_i=2k+2}
 \|\Pi_i\nabla^{j_i}u\varphi_k^2\|_{L^1(B)}
 +C\|\partial^k_{\phi}u\nabla\varphi_k\|^2_{L^2(B)}\\
 &\le\frac14\|\nabla\partial^k_{\phi}u\varphi_k\|^2_{L^2(B)}
 +C\sum_{{1\le j_i\le k}\atop{\Sigma_ij_i=2k+2}}\|\Pi_i\nabla^{j_i}u\varphi_k^2\|_{L^1(B)}
 +C\|\partial^k_{\phi}u\nabla\varphi_k\|^2_{L^2(B)},
 \end{split}
\end{equation*}
as claimed. Finally, with 
\begin{equation*}
  \nu(u)\cdot u_{\phi r}=(dist_N(u))_{\phi r}-u_r\cdot d\nu(u)u_{\phi}
\end{equation*}
as in the proof of Lemma \ref{lemma4.3}, for $j=0$ we can write 
\begin{equation*}
 \nu(u)\cdot\partial^k_{\phi}u_r
 =\partial^{k-1}_{\phi}\big(\nu(u)\cdot u_{\phi r}\big)+II 
 =\partial^k_{\phi}(dist_N(u))_r+III,
\end{equation*}
where the terms in $II$ and $III$ involve products of at least two derivatives of 
orders between $1$ and $k$ of $u$. Thus we have
\begin{equation*}
 \begin{split}
   I_0&=(\partial^k_{\phi}u\cdot\nu(u)\varphi_k^2,
   \nu(u)\cdot\partial^k_{\phi}u_r)_{L^2(\partial B)}\\
   &=(\partial^k_{\phi}u\cdot\nu(u)\varphi_k^2,\partial^k_{\phi}(dist_N(u))_r)_{L^2(\partial B)}
   +II_0
 \end{split}
\end{equation*}
with a term $II_0$ that can be dealt with in the same way as the terms $I_j$, $1\le j\le k$.

Using the divergence theorem and integrating by parts we can write the leading term as
\begin{equation*}
 \begin{split}
   \hat{I}_0&
   :=(\partial^k_{\phi}u\cdot\nu(u)\varphi_k^2,\partial^k_{\phi}(dist_N(u))_r)_{L^2(\partial B)}\\
   &=\big(\nabla\big(\partial^k_{\phi}u\cdot\nu(u)\varphi_k^2\big),
   \nabla\partial^k_{\phi}(dist_N(u))\big)_{L^2(B)}\\
   &\qquad+\big(\partial^k_{\phi}u\cdot\nu(u)\varphi_k^2,
   \Delta\partial^k_{\phi}(dist_N(u))\big)_{L^2(B)}\\
   &=\big(\nabla\big(\partial^k_{\phi}u\cdot\nu(u)\varphi_k^2\big),
   \nabla\partial^k_{\phi}(dist_N(u))\big)_{L^2(B)}\\
   &\qquad-\big(\partial_{\phi}\big(\partial^k_{\phi}u\cdot\nu(u)\varphi_k^2\big),
   \Delta\partial^{k-1}_{\phi}(dist_N(u))\big)_{L^2(B)}
 \end{split}
\end{equation*}
to see that this term may be bounded
\begin{equation*}
 \begin{split}
   |\hat{I}_0|&\le C\|(|\nabla\partial^k_{\phi}u|+|\partial^k_{\phi}u\nabla u|)\varphi_k
   +|\partial^k_{\phi}u\nabla\varphi_k|\|_{L^2(B)}
   \|\nabla^{k+1}(dist_N(u))\varphi_k\|_{L^2(B)}.
 \end{split}
\end{equation*}
But by elliptic regularity we again have 
\begin{equation*}
 \begin{split}
  \|&\nabla^{k+1}(dist_N(u))\varphi_k\|_{L^2(B)}\\
  &\le\|\nabla^{k+1}(dist_N(u)\varphi_k)\|_{L^2(B)}
  +C\sum_{1\le j\le k+1}\|\nabla^{k+1-j}(dist_N(u))\nabla^j\varphi_k\|_{L^2(B)}\\
  &\le C\|\Delta(dist_N(u)\varphi_k)\|_{H^{k-1}(B)}
  +C\sum_{1\le j\le k+1}\|\nabla^{k+1-j}(dist_N(u))\nabla^j\varphi_k\|_{L^2(B)},
 \end{split}
\end{equation*}
where from \eqref{3.5} we can bound the first term on the right
\begin{equation*}
 \begin{split}
  \|\Delta&(dist_N(u))\varphi_k\|_{H^{k-1}(B)}
  \le\sum_{0\le j<k}\|\nabla^j\big(\nabla u\cdot d\nu(u)\nabla u\varphi_k\big)\|_{L^2(B)}\\
  &\le C\sum_{0\le j_0<k,\,1\le j_i\le k,\,\Sigma_{i\ge 0}j_i\le k+1}
  \|\Pi_i\nabla^{j_i}u\nabla^{j_0}\varphi_k\|_{L^2(B)}.
 \end{split}
\end{equation*}
Moreover, using that $dist_N(u))=0$ on $\partial B$, with the help of Poincar\'e's inequality
we find the bound 
\begin{equation*}
 \begin{split}
  \|dist_N(u)\nabla^{k+1}\varphi_k\|^2_{L^2(B)}
  \le CR_k^{-2k}\|\nabla(dist_N(u))\|^2_{L^2(B_{R_k}(z_0))}\le CR_0^{-2k}E(u).
 \end{split}
\end{equation*}
The remaining terms for $1\le j\le k$ can be estimated
\begin{equation*}
 \begin{split}
  \|\nabla^{k+1-j}(dist_N(u))\nabla^j\varphi_2\|_{L^2(B)}
  \le C\sum_{1\le j_i\le k,\,\Sigma_ij_i=k+1-j}\|\Pi_i\nabla^{j_i}u\nabla^j\varphi_2\|_{L^2(B)}
 \end{split}
\end{equation*}
via the chain rule. Thus, finally, we obtain the bound
\begin{equation*}
 \begin{split}
  \|\nabla^{k+1}&(dist_N(u))\varphi_k\|_{L^2(B)}\\
  &\le C\sum_{1\le j_0,j_i\le k,\,\Sigma_{i\ge 0}j_i\le k+1}
  \|\Pi_{i>0}\nabla^{j_i}u\nabla^{j_0}\varphi_2\|_{L^2(B)}+CR_0^{-2k}E(u_0).
 \end{split}
\end{equation*}
By Cauchy-Schwarz and Young's inequality thus we can bound
\begin{equation*}
 \begin{split}
 |\hat{I}_0|&\le\frac14\|\nabla\partial^k_{\phi}u\varphi_k\|^2_{L^2(B)}
 +C\|\partial^k_{\phi}u\nabla u\varphi_k\|^2_{L^2(B)}\\
 &\qquad+C\sum_{1\le j_i\le k,\,\Sigma_ij_i\le k+1}
  \|\Pi_{i>0}\nabla^{j_i}u\nabla^{j_0}\varphi_k\|^2_{L^2(B)}+CR_0^{-2k}E(u_0),
 \end{split}
\end{equation*}
and together with our above estimate for the terms $I_j$, $j\ge 1$, our claim follows.
\end{proof}

\begin{proposition}\label{prop4.10}
There exists a constant $\delta>0$ independent of $R_0>0$ 
such that whenever \eqref{4.10} holds then for any $T\le T_0$ with a constant 
$C_2=C_2(T,R,u_0)>0$ bounded by the terms on the right hand 
side in the statement of Proposition \ref{prop4.8} there holds the estimate
\begin{equation*}
 \begin{split}
  \sup_{0<t<T}\int_{\partial B}|u_{\phi\phi}(t)|^2\varphi_2^2\,d\phi
  +\int_0^T\int_B|\nabla&u_{\phi\phi}|^2\varphi_2^2\,dz\,dt\\
  &\le C_2\int_{\partial B}|u_{0,\phi\phi}|^2\varphi_2^2\,d\phi+C_2.
 \end{split}
\end{equation*}
\end{proposition}

\begin{proof}
For $k=2$ with the help of Young's inequality  we can bound 
\begin{equation*}
 \begin{split}
 J_1&=\sum_{1\le j_i\le k,\, \Sigma_ij_i\le 2k+2}\|\Pi_i\nabla^{j_i}u\varphi_k^2\|_{L^1(B)}\\
  &\quad\le C\|(|\nabla^2u|^3+|\nabla^2u|^2|\nabla u|^2+|\nabla^2u||\nabla u|^4
 +|\nabla u|^6+1)\varphi_2^2\|_{L^1(B)}\\
&\quad\le C\|(|\nabla^2u|^3+|\nabla u|^6+1)\varphi_2^2\|_{L^1(B)},
 \end{split}
\end{equation*}
and
\begin{equation*}
 \begin{split}
 J_2&=\sum_{1\le j_0,j_i\le k,\,\Sigma_{i\ge 0}j_i\le k+1}
  \|\Pi_{i>0}\nabla^{j_i}u\nabla^{j_0}\varphi_2\|^2_{L^2(B)}\\
 &\le C\|(|\nabla^2u|^2+|\nabla u|^4+1)|\nabla\varphi_2|^2
 +(|\nabla u|^2+1)|\nabla^2\varphi_2|^2\|_{L^1(B)}.
 \end{split}
\end{equation*}

Observing that $\varphi_1=1$ on the support of $\varphi_2$, by \eqref{A.2} for the first term
in $J_1$ we have
\begin{equation*}
 \begin{split}
 \||\nabla^2&u|^3\varphi_2^2\|_{L^1(B)}\le\|\nabla^2 u\varphi_2\|^2_{L^4(B)}
 \|\nabla^2 u\varphi_1\|_{L^2(B)}\\
 &\le C\|\nabla^2 u\varphi_2\|_{H^1(B)}\|\nabla^2 u\varphi_2\|_{L^2(B)}
 \|\nabla^2 u\varphi_1\|_{L^2(B)}\\
 &\le C(\|\nabla^3 u\varphi_2\|_{L^2(B)}
 +\|\nabla^2 u\varphi_1\|_{L^2(B)})
 \|\nabla^2 u\varphi_2\|_{L^2(B)}\|\nabla^2 u\varphi_1\|_{L^2(B)}.
 \end{split}
\end{equation*}
Moreover, arguing as in \eqref{A.1} for the function $|\nabla u|^6\varphi_2^2$ in place of 
$|v|^4\varphi^2$, we can bound
\begin{equation*}
 \begin{split}
 \int_B|\nabla&u|^6\varphi_2^2dz
 \le C\Big(\int_B\big(|\nabla^2 u||\nabla u|^2\varphi_2
 +|\nabla u|^3|\nabla\varphi_2|\big)dz\Big)^2\\
 &\le C\Big(\int_B|\nabla^2 u|^3\varphi_2^2dz\Big)^{2/3}
 \Big(\int_B|\nabla u|^3\varphi_2^{1/2}dz\Big)^{4/3}
 +C\Big(\int_B|\nabla u|^3|\nabla\varphi_2|\big)dz\Big)^2,
 \end{split}
\end{equation*}
where by H\"older's inequality we have
\begin{equation*}
 \begin{split}
 \int_B&|\nabla u|^3\varphi_2^{1/2}dz
 \le\Big(\int_B|\nabla u|^6\varphi_2^2dz\Big)^{1/4}
 \Big(\int_B|\nabla u|^2\varphi_1^2dz\Big)^{3/4}
 \end{split}
\end{equation*}
so that with Young's inequality we obtain
\begin{equation*}
 \begin{split}
 \int_B|\nabla&u|^6\varphi_2^2dz
 \le C\delta\Big(\int_B|\nabla^2 u|^3\varphi_2^2dz\Big)^{2/3}
 \Big(\int_B|\nabla u|^6\varphi_2^2dz\Big)^{1/3}\\
 &\qquad+C\Big(\int_B|\nabla u|^3|\nabla\varphi_2|\big)dz\Big)^2\\
 &\le\frac12\int_B|\nabla u|^6\varphi_2^2dz
 +C\int_B|\nabla^2 u|^3\varphi_2^2dz
 +C\Big(\int_B|\nabla u|^3|\nabla\varphi_2|\big)dz\Big)^2.
 \end{split}
\end{equation*}
With Young's inequality for suitable $\varepsilon>0$, and using \eqref{4.9}, 
we then can bound
\begin{equation*}
 \begin{split}
 J_1&\le C\|(|\nabla^2u|^3+1)\varphi_2^2\|_{L^1(B)}
 +C\||\nabla u|^3|\nabla\varphi_2|\|^2_{L^1(B)}
 \le\varepsilon\|\nabla^3 u\varphi_2\|^2_{L^2(B)}\\
 &\qquad\qquad+C(1+\|\nabla^2 u\varphi_2\|^2_{L^2(B)})
 \|\nabla^2 u\varphi_1\|^2_{L^2(B)}+C\||\nabla u|^3|\nabla\varphi_2|\|^2_{L^1(B)}\\
 &\le\frac12\|\nabla\partial^2_{\phi}u\varphi_2\|^2_{L^2(B)}
 +C(1+\|\nabla\partial_{\phi}u\varphi_2\|^2_{L^2(B)})
 \|\nabla\partial_{\phi}u\varphi_1\|^2_{L^2(B)}+C_1,
 \end{split}
\end{equation*}
where we also have estimated 
\begin{equation*}
 \begin{split}
 \||\nabla u|^3&|\nabla\varphi_2|\|^2_{L^1(B)}
 \le C\|\nabla u\varphi_1\|^4_{L^4(B)}\|\nabla u\varphi_1\|^2_{L^2(B)}\\
 &\le C\big(\|\nabla^2u\varphi_1\|^2_{L^2(B)}+E(u)\big)\|\nabla u\varphi_1\|^4_{L^2(B)}
 \le C\|\nabla\partial_{\phi}u\varphi_1\|^2_{L^2(B)}+C.
 \end{split}
\end{equation*}

Similarly, with \eqref{A.2} we have
\begin{equation*}
 J_2\le C\|\nabla^2 u\varphi_1\|^2_{L^2(B)}+C.
\end{equation*}
Thus, from Lemma \ref{lemma4.7} we obtain
\begin{equation}\label{4.11}
 \begin{split}
  \frac{d}{dt}\big(\|\partial^2_{\phi}&u\varphi_2\|^2_{L^2(\partial B)}\big)
  +\frac12\|\nabla\partial^2_{\phi}u\varphi_2\|^2_{L^2(B)}\\
  &\le C(1+\|\nabla\partial_{\phi}u\varphi_2\|^2_{L^2(B)})
 \|\nabla\partial_{\phi}u\varphi_1\|^2_{L^2(B)}+C.
 \end{split}
\end{equation}

Denote as $C_1=C_1(T,R,u_0)>0$ a constant bounded by the terms on the right hand 
side in the statements of Propositions \ref{prop4.8}.
By elliptic regularity, using that $|\Delta(u\varphi_2|\le 2|\nabla u\nabla\varphi_2|+C$
we can bound 
\begin{equation*}
 \begin{split}
 \|\nabla^2&u\varphi_2\|^2_{L^2(B)}
 \le\|u\varphi_2\|^2_{H^2(B)}+C\|\nabla u\nabla\varphi_2\|^2_{L^2(B)}+C\\
 &\le C\|u\varphi_2\|^2_{H^2(\partial B)}+\|\Delta(u\varphi_2)\|^2_{L^2(B)}
 +C\|\nabla u\nabla\varphi_2\|^2_{L^2(B)}+C\\
 &\le C\|\partial^2_{\phi}u\varphi_2\|^2_{L^2(\partial B)}+CE(u)+C_1.
 \end{split}
\end{equation*}
From \eqref{4.11} we then obtain the differential inequality
\begin{equation*}
 \begin{split}
 \frac{d}{dt}\big(1+\|\partial^2_{\phi}&u\varphi_2\|^2_{L^2(\partial B)}\big)
 \le C(1+\|\partial^2_{\phi}u\varphi_2\|^2_{L^2(\partial B)})
 \|\nabla\partial_{\phi}u\varphi_1\|^2_{L^2(B)}+C_1;
 \end{split}
\end{equation*}
that is,
\begin{equation*}
 \begin{split}
 \frac{d}{dt}\Big(\log\big(1+\|\partial^2_{\phi}&u\varphi_2\|^2_{L^2(\partial B)}\big)\Big)
 \le C\|\nabla\partial_{\phi}u\varphi_1\|^2_{L^2(B)}+C_1,
 \end{split}
\end{equation*}
and the right hand side is integrable in time by Proposition \ref{prop4.8}.
The claim follows.
\end{proof}

We continue by induction. 

\begin{proposition}\label{prop4.11}
There exists a constant $\delta>0$ independent of $R_0>0$ with the following property.
Whenever \eqref{4.10} holds, then for any $k\ge 3$, any smooth $u_0\in H^{1/2}(S^1;N)$, 
and any $T<T_0$, there holds
\begin{equation*}
  \sup_{0<t<T}\int_{\partial B}|\partial^k_{\phi}u(t)|^2\varphi_k^2d\phi
  +\int_0^T\int_B|\nabla\partial^k_{\phi}u|^2\varphi_k^2dz\,dt
  \le C_k\int_{\partial B}|\partial^k_{\phi}u_0|^2\varphi_k^2d\phi+C_k,
\end{equation*}
where we denote as $C_k=C_k(T,R,u_0)>0$ a constant bounded by the terms on the right hand 
side in the statement of the proposition for $k-1$.
\end{proposition}

\begin{proof}
By Proposition \ref{prop4.10} the claimed result holds true for $k=2$. 
Suppose the claim holds true for some $k_0\ge 2$ and let $k=k_0+1$. 
Note that by elliptic regularity, as in the proof of Proposition \ref{prop4.10}
we can bound 
\begin{equation*}
 \begin{split}
 \|\nabla^k&u\varphi_k\|^2_{L^2(B)}
 \le\|u\varphi_k\|^2_{H^k(B)}+C\sum_{j<k}\|\nabla^ju\nabla^{k-j}\varphi_k\|^2_{L^2(B)}\\
 &\le C\|u\varphi_k\|^2_{H^k(\partial B)}+C\|\Delta(u\varphi_k)\|^2_{H^{k-2}(B)}
 +C\sum_{j<k}\|\nabla^ju\nabla^{k-j}\varphi_k\|^2_{L^2(B)}\\
 &\le C\|\partial^k_{\phi}u\varphi_k\|^2_{L^2(\partial B)}
 +C\sum_{j<k}\|\nabla^ju\nabla^{k-j}\varphi_k\|^2_{L^2(B)}+C_k.
 \end{split}
\end{equation*}
By induction hypothesis and Sobolev's embedding
$H^2(B)\hookrightarrow W^{1,4}\cap C^0(\bar{B})$ for $0\le t<T$ we then have the uniform bounds 
\begin{equation*}
 \begin{split}
  \|\nabla^{k_0}u\varphi_{k_0}\|^2_{L^2(B)}&+\|\nabla^{k_0-1}u\varphi_{k_0}\|^2_{L^4(B)}
  +\sum_{j=1}^{k_0-2}\|\nabla^{j}u\varphi_{k_0}\|^2_{L^{\infty}(B)}\le C_k,
 \end{split}
\end{equation*}
and it follows that 
\begin{equation*}
 \begin{split}
 \|\nabla^k&u\varphi_k\|^2_{L^2(B)}+\|\nabla^{k_0}u\varphi_k\|^2_{L^4(B)}
  +\|\nabla^{k_0-1}u\varphi_k\|^2_{L^{\infty}(B)}
 \le C\|\partial^k_{\phi}u\varphi_k\|^2_{L^2(\partial B)}+C_k.
 \end{split}
\end{equation*}

Again let 
\begin{equation*}
 \begin{split}
  J_1&:=\sum_{1\le j_i\le k,\, \Sigma_ij_i=2k+2}
  \|\Pi_i\nabla^{j_i}u\varphi_k^2\|_{L^1(B)}\\
  &\le\|\big(|\nabla^ku|^2(|\nabla^2u|+|\nabla u|^2)
  +|\nabla^ku||\nabla^{k_0}u||\nabla^3u|+\dots+|\nabla u|^{2k+2}\big)
  \varphi_k^2\|_{L^1(B)}.
 \end{split}
\end{equation*}
and set 
\begin{equation*}
 \begin{split}
 J_2&=\sum_{1\le j_0,j_i\le k,\,\Sigma_{i\ge 0}j_i\le k+1}
  \|\Pi_{i>0}\nabla^{j_i}u\nabla^{j_0}\varphi_2\|^2_{L^2(B)}.
 \end{split}
\end{equation*}

Suppose $k_0=2$. Recalling that $\varphi_k=\varphi_k\varphi_{k_0}$, 
we can bound the listed terms
\begin{equation*}
 \begin{split}
  \||\nabla^3&u|^2(|\nabla^2u|+|\nabla u|^2)\varphi_3^2\|_{L^1(B)}\\
  &\le\|\nabla^3u\varphi_3\|^2_{L^4(B)}(\|\nabla^2u\varphi_2\|_{L^2(B)}
  +\|\nabla u\varphi_2\|^2_{L^4(B)})\\
  &\le C_3\|\nabla\partial^3_{\phi}u\varphi_3\|_{L^2(B)}\|\nabla^3u\varphi_3\|_{L^2(B)}
  + C_3\|\nabla^3u\varphi_2\|^2_{L^2(B)}+C_3\\
  &\le C_3\|\nabla\partial^3_{\phi}u\varphi_3\|_{L^2(B)}
  \|\partial^3_{\phi}u\varphi_3\|_{L^2(\partial B)}
  + C_3\|\nabla\partial^2_{\phi}u\varphi_2\|^2_{L^2(B)}+C_3\\
  &\le\varepsilon\|\nabla\partial^3_{\phi}u\varphi_3\|^2_{L^2(B)}
  +C_3\|\partial^3_{\phi}u\varphi_3\|^2_{L^2(\partial B)}
  + C_3\|\nabla\partial^2_{\phi}u\varphi_2\|^2_{L^2(B)}+C_3,
 \end{split}
\end{equation*}
and 
\begin{equation*}
 \begin{split}
  \||\nabla u|^8\varphi_3^2\|_{L^1(B)}
  &\le\|\nabla u\varphi_3\|^2_{L^{\infty}(B)}\|\nabla u\varphi_2\|^6_{L^6(B)}\\
  &\le C_3\|\partial^3_{\phi}u\varphi_3\|^2_{L^2(\partial B)}+C_3,
 \end{split}
\end{equation*}
respectively. 
Here we also have used \eqref{A.1}, \eqref{A.2} to bound  
\begin{equation*}
 \begin{split}
  \|\nabla u&\varphi_2\|^3_{L^6(B)}\le\|\nabla(|\nabla u|^3\varphi^3_2)\|_{L^1(B)}\\
  &\le C\|(|\nabla^2u|\varphi_2
  +|\nabla u||\nabla\varphi_2|)|\nabla u|^2\varphi^2_2\|_{L^1(B)}\\
  &\le C(\|\nabla^2u\varphi_2\|_{L^2(B)}+\|\nabla u\nabla\varphi_2\|_{L^2(B)})
  \|\nabla u\varphi_2\|^2_{L^4(B)}\\
  &\le C(\|\nabla^2u\varphi_2\|_{L^2(B)}+\|\nabla u\nabla\varphi_2\|_{L^2(B)})^2
  \|\nabla u\varphi_2\|_{L^2(B)}\le C_3.
 \end{split}
\end{equation*}
Similarly, we can bound the remaining terms and the terms in $J_2$ to obtain
\begin{equation*}
 \begin{split}
  \frac{d}{dt}\big(\|\partial^3_{\phi}&u\varphi_3\|^2_{L^2(\partial B)}\big)
  +\frac12\|\nabla\partial^3_{\phi}u\varphi_3\|^2_{L^2(B)}\\
  &\le C_3(1+\|\partial^3_{\phi}u\varphi_3\|^2_{L^2(\partial B)})
  (1+\|\nabla\partial^2_{\phi}u\varphi_2\|^2_{L^2(B)})+C_3
 \end{split}
\end{equation*}
from Lemma \ref{lemma4.9} and then 
\begin{equation*}
 \begin{split}
 \frac{d}{dt}\Big(\log\big(1+\|\partial^3_{\phi}&u\varphi_2\|^2_{L^2(\partial B)}\big)\Big)
 \le C_3(1+\|\nabla\partial^2_{\phi}u\varphi_2\|^2_{L^2(B)}),
 \end{split}
\end{equation*}
where the right hand side is integrable in time by Proposition \ref{prop4.10}.
The claim for $k=3$ thus follows.

For $k\ge 4$ the analysis is similar (but simpler) and may be left to the reader.
\end{proof}

\section{Local existence}\label{Local existence}
In order to show local existence we approximate the flow equation \eqref{1.3}
by the equation 
\begin{equation}\label{5.1}
   u_t=-(\varepsilon+d\pi_N(u)) u_r\hbox{ on } \partial B.
\end{equation}
where $\varepsilon>0$ and where we smoothly extend the nearest-neighbor projection $\pi_N$,
originally defined only in the $\rho$-neighborhood $N_{\rho}$ of $N$, to the whole ambient $\R^n$.
Our aim then is to show that for given smooth initial data $u_0$ 
the evolution problem \eqref{5.1}, \eqref{1.4} admits a smooth solution $u_{\varepsilon}$ 
which remains uniformly smoothly bounded on a uniform time interval as 
$\varepsilon\downarrow 0$.
Fixing some $0<\varepsilon<1/2$, we show existence for the problem 
\eqref{5.1} with data \eqref{1.4} by means of a fixed-point argument.

To set up the argument, fix smooth initial data $u_0\colon S^1\to N$
with harmonic extension $u_0\in C^{\infty}(\bar{B};\R^n)$ and some $k\ge 2$. 
For suitable $T>0$ to be determined let 
\begin{equation*}
   X=L^{\infty}\big([0,T];H^{k+1}(B;\R^n)\big)\cap H^1(S^1\times [0,T];\R^n)
\end{equation*}
and set 
\begin{equation*}
 \begin{split}
   V=\{v & \in X;\;v(0)=u_0,\ 
   \Delta v(t)=0\hbox{ in } B \hbox{ for } 0\le t\le T,\\ 
   &\|v\|^2_X=\sup_{0\le t\le T}\|v(t)\|^2_{H^{k+1}(B)}
   +\int_0^T\int_{S^1}|v_t|^2d\phi\,dt\le 4R_0^2\},
 \end{split}
\end{equation*}
where $R_0=\|u_0\|_{H^{k+1}(B)}$. 
We endow the space $V$ with the metric derived from the semi-norm 
\begin{equation*}
   |v|^2_X=\sup_{0\le t\le T}\|\nabla v(t)\|^2_{L^2(B)}+
   \int_0^T\int_{S^1}|v_t|^2d\phi\,dt.
\end{equation*}
Note that this metric is positive definite on $V$ in view of the initial 
condition that we impose.

\begin{lemma}\label{lemma5.1}
$V$ is a complete metric space.
\end{lemma}

\begin{proof} 
Let $(v_m)_{m\in\N}\subset V$ with $|v_l-v_m|_X\to 0$ ($l,m\to\infty$).
By the theorem of Banach-Alaoglu a subsequence $v_m\rightharpoondown v$ 
weakly-$*$ in $L^{\infty}\big([0,T];H^{k+1}(B)\big)$ 
with $v_{m,t}\to v_t$ weakly in $L^2([0,T]\times S^1)$, 
and by weak lower semi-continuity of the norm there holds 
\begin{equation*}
   \|v\|^2_X\le\limsup_{m\to\infty}\|v_m\|^2_X\le 4R_0^2.
\end{equation*}
Moreover, we have $\Delta v(t)=0$ for all $0\le t\le T$ and $v(0)=u_0$ by 
compactness of the trace operator $H^1(S^1\times [0,T])\ni u\mapsto u(0)\in L^2(S^1)$.
Hence $v\in V$. 

Moreover, we have 
\begin{equation*}
  |v_l-v|_X\le\limsup_{m\to\infty}|v_l-v_m|_X\to 0\ \hbox{ as }l\to\infty.
\end{equation*}
\end{proof}

\begin{lemma}\label{lemma5.2}
There is $T_2>0$ such that for any $T\le T_2$, any $v\in V$ there is a solution 
$u=\Phi(v)\in V$ of the equation
\begin{equation}\label{5.2}
   u_t=-(\varepsilon+d\pi_N(v)) u_r\hbox{ on } \partial B\times[0,T_2[,
\end{equation}
satisfying \eqref{1.4}.
\end{lemma}

\begin{proof} 
For $v\in V$ we construct a solution $u=\Phi(v)\in X$ of \eqref{5.2} via Galerkin 
approximation. For this let $(\varphi_l)_{l\in\N_0}$ be Steklov eigenfunctions 
of the Laplacian, satisfying
\begin{equation*}
   \Delta\varphi_l=0 \hbox{ in } B
\end{equation*}
with boundary condition
\begin{equation*}
   \partial_r\varphi_l=\lambda_l\varphi_l \hbox{ on } \partial B,\ l\in\N_0.
\end{equation*}
Note that the Steklov eigenvalues are given by 
$\lambda_0=0$ and $\lambda_{2l-1}=\lambda_{2l}=l$, $l\in\N$. In fact,
we may choose $\varphi_0\equiv 1/\sqrt{2\pi}$ and 
\begin{equation}\label{5.3}
   \varphi_{2l-1}(re^{i\theta})=\frac{1}{\sqrt{\pi}}r^lsin(l\theta),\
   \varphi_{2l}(re^{i\theta})=\frac{1}{\sqrt{\pi}}r^lcos(l\theta),\ l\in\N.
\end{equation}
to obtain an orthonormal basis for $L^2(S^1)$ consisting of these functions.
Given $m\in\N$ then let $u^{(m)}(t,z)=\sum_{l=0}^ma^{(m)}_l(t)\varphi_l(z)$ solve the system of 
equations
\begin{equation}\label{5.4}
 \begin{split}
   \partial_ta^{(m)}_l=(\varphi_l,&u^{(m)}_t)_{L^2(S^1)}
   =-\big(\varphi_l,(\varepsilon+d\pi_N(v))u^{(m)}_r\big)_{L^2(S^1)}\\
   &=-\sum_{j=0}^ma^{(m)}_j\lambda_j
   \big(\varphi_l,(\varepsilon+d\pi_N(v))\varphi_j\big)_{L^2(S^1)},\  0\le l\le m.
 \end{split}
\end{equation}
Since for any $m\in\N$ the coefficients 
$\lambda_j(\varphi_l,(\varepsilon+d\pi_N(v))\varphi_j\big)_{L^2(S^1)}$
of this system are uniformly bounded for any $v\in V$, for any $m\in\N$ 
there exists a unique global solution 
$a^{(m)}=(a^{(m)}_l)_{0\le l\le m}$ of \eqref{5.4} with initial data 
$a^{(m)}_l(0)=a_{l0}=(u_0,\varphi_l)_{L^2(S^1)}$, $0\le l\le m$.

Note that for any $m\in\N$ and any $j\in\N_0$ the function 
\begin{equation*}
   \partial^{2j}_{\phi}(ru^{(m)}_r)\in span\{\varphi_l;\;0\le l\le m\},
\end{equation*}
and $\partial^{2j}_{\phi}u^{(m)}$ is harmonic. In particular, for $j=0$ we obtain 
\begin{equation}\label{5.5}
 \begin{split}
 \frac12\frac{d}{dt}\big(\|\nabla&u^{(m)}\|^2_{L^2(B)}\big)
 =\int_B\nabla u^{(m)}\nabla u^{(m)}_t\;dz=(u^{(m)}_r,u^{(m)}_t)_{L^2(S^1)}\\
 &=-(u^{(m)}_r,(\varepsilon+d\pi_N(v)) u^{(m)}_r)_{L^2(S^1)}\\
 &=-\varepsilon\|u^{(m)}_r\|^2_{L^2(S^1)}-\|d\pi_N(v)u^{(m)}_r\|^2_{L^2(S^1)}\\
 &\le-\frac12\|u^{(m)}_t\|^2_{L^2(S^1)}\le 0,
 \end{split}
\end{equation}
and we find the uniform $H^1$-bound 
\begin{equation}\label{5.6}
 \begin{split}
  \sup_{t\ge 0}&\|\nabla u^{(m)}(t)\|^2_{L^2(B)}
  +\varepsilon\|u^{(m)}_r\|^2_{L^2([0,\infty[\times S^1)}
  +\|u^{(m)}_t\|^2_{L^2([0,\infty[\times S^1)}\\
  &\le 2\|\nabla u^{(m)}(0)\|^2_{L^2(B)}\le 2\|\nabla u_0\|^2_{L^2(B)}\le 2R_0^2.
 \end{split}
\end{equation}

Moreover, for $j=k\in\N$ as in the definition of $X$ upon integrating by parts we find 
\begin{equation}\label{5.7}
 \begin{split}
 \frac12&\frac{d}{dt}\big(\|\nabla\partial^k_{\phi}u^{(m)}\|^2_{L^2(B)}\big)
 =(-1)^k\int_B\nabla\partial^{2k}_{\phi}u^{(m)}\nabla u^{(m)}_t\;dz\\
 &=(-1)^k(\partial^{2k}_{\phi}u^{(m)}_r,u^{(m)}_t)_{L^2(S^1)}\\
 &=(-1)^{k+1}(\partial^{2k}_{\phi}u^{(m)}_r,
 (\varepsilon+d\pi_N(v))u^{(m)}_r)_{L^2(S^1)}\\
 &=-\varepsilon\|\partial^k_{\phi}u^{(m)}_r\|^2_{L^2(S^1)}
 -\|d\pi_N(v)\partial^k_{\phi}u^{(m)}_r\|^2_{L^2(S^1)}+I,
 \end{split}
\end{equation}
where $I=\sum_{j=1}^k\Big({k\atop j}\Big)I_j$ with
\begin{equation*}
 I_j=-(\partial^k_{\phi}u^{(m)}_r,\partial^j_{\phi}(d\pi_N(v))
   \partial^{k-j}_{\phi}u^{(m)}_r)_{L^2(S^1)}
\end{equation*}
similar to the proof of Lemma \ref{lemma4.3}. However, now we simply bound 
\begin{equation*}
 |I_j|\le C\sum_{\Sigma_ij_i=j}\|\partial^k_{\phi}u^{(m)}_r\|_{L^2(S^1)}
 \|\Pi_i\partial^{j_i}_{\phi}v\partial^{k-j}_{\phi}u^{(m)}_r\|_{L^2(S^1)}, \ 1\le j\le k.
\end{equation*}

Note that by compactness of Sobolev's embedding $H^1(S^1)\hookrightarrow L^{\infty}(S^1)$
and Ehrlich's lemma for any number $1\le j\le k$, any $\delta>0$ we can bound 
\begin{equation*}
 \begin{split}
  \|\partial^{k-j}_{\phi}&u^{(m)}_r\|_{L^{\infty}(S^1)}
  \le\delta\|\partial^{k-j+1}_{\phi}u^{(m)}_r\|_{L^2(S^1)}
  +C(\delta)\|\partial^{k-j}_{\phi}u^{(m)}_r\|_{L^2(S^1)}\\
  &\le2\delta\|\partial^k_{\phi}u^{(m)}_r\|_{L^2(S^1)}
  +C(\delta)\|u^{(m)}_r\|_{L^2(S^1)}.
 \end{split}
\end{equation*}
On the other hand, for any $v\in V$ by the trace theorem we have
\begin{equation*}
  \|\partial^k_{\phi}v\|_{L^2(S^1)}\le C\|\partial^k_{\phi}v\|_{H^1(B)}
  \le C\|v\|_{H^{k+1}(B)}\le CR_0
\end{equation*}
and we therefore also can bound 
\begin{equation*}
  \|\partial^j_{\phi}v\|_{L^{\infty}(S^1)}
  \le C\|\partial^k_{\phi}v\|_{L^2(S^1)}+\|\partial^j_{\phi}v\|_{L^2(S^1)}
  \le C\|v\|_{H^{k+1}(B)}\le CR_0.
\end{equation*}
for any $1\le j<k$.

Thus, for sufficiently small $\delta>0$ with a constant $C>0$ depending on 
$\varepsilon>0$ and $R_0$ there holds
\begin{equation*}
  |I|\le\varepsilon/2 \|\partial^k_{\phi}u^{(m)}_r\|^2_{L^2(S^1)}
  +C\|u^{(m)}_r\|^2_{L^2(S^1)}
\end{equation*}
and from \eqref{5.7} with the help of \eqref{3.7} we obtain the inequality
\begin{equation*}
 \begin{split}
 \frac{d}{dt}&\big(\|\nabla\partial^k_{\phi}u^{(m)}\|^2_{L^2(B)}\big)
 \le C\|u^{(m)}_r\|^2_{L^2(S^1)}=C\|u^{(m)}_{\phi}\|^2_{L^2(S^1)}
 \le C\|u^{(m)}_{\phi}\|^2_{H^1(B)}\\
 &\le C\|\nabla\partial^k_{\phi}u^{(m)}\|^2_{L^2(B)}+C\|\nabla u^{(m)}\|^2_{L^2(B)}
 \le C(1+\|\nabla\partial^k_{\phi}u^{(m)}\|^2_{L^2(B)}),
 \end{split}
\end{equation*}
where we recall \eqref{5.6} for the last conclusion. 

It follows that for suitably small $T>0$ there holds $\|u^{(m)}\|^2_X\le 4R_0^2$
for all $m\in\N$. Thus, there is a sequence $m\to\infty$ such that 
$u^{(m)}\rightharpoondown u$ weakly-$*$ in $L^{\infty}([0,T]; H^{k+1}(B))$ 
with $u^{(m)}_t\rightharpoondown u_t$ weakly in $L^2([0,T]\times S^1)$, 
where $u=:\Phi(v)\in V$ solves equation \eqref{5.2}.
\end{proof}

\begin{lemma}\label{lemma5.3}
There is $T>0$ such that for $v_1,v_2\in V$ there holds
\begin{equation*}
    |\Phi(v_1)-\Phi(v_2)|_X\le\frac12|v_1-v_2|_X.
\end{equation*}
\end{lemma}

\begin{proof} 
Let $T_2>0$ be as determined in Lemma \ref{lemma5.2} and fix some $0<T\le T_2$. 
For $v_1,v_2\in V$ then we have $u_i=:\Phi(v_i)\in V$, $i=1,2$. Set
$w=u_1-u_2$, $v=v_1-v_2$, and compute
\begin{equation}\label{5.8}
   w_t=-(\varepsilon+d\pi_N(v_1))w_r
   -(d\pi_N(v_1)-d\pi_N(v_2))u_{2,r}
  \hbox{ on } \partial B=S^1.
\end{equation}
Multiplying with $w_r$ and integrating we obtain 
\begin{equation*}
 \begin{split}
 \frac12&\frac{d}{dt}\big(\|\nabla w\|^2_{L^2(B)}\big)
 =\int_B\nabla w\nabla w_t\;dx=(w_r,w_t)_{L^2(S^1)}
 =-\varepsilon\|w_r\|^2_{L^2(S^1)}\\
 &-\|d\pi_N(v_1)w_r\|^2_{L^2(S^1)}
 -(w_r,(d\pi_N(v_1)-d\pi_N(v_2)u_{2,r})_{L^2(S^1)},
 \end{split}
\end{equation*}
where with $\|u_{2,r}\|_{L^{\infty}(S^1)}\le C\|u_2\|_{H^3(B)}\le CR_0$ we can bound
\begin{equation*}
 \begin{split}
 |(w_r,&(d\pi_N(v_1)-d\pi_N(v_2))u_{2,r})_{L^2(S^1)}|
 \le C\|w_r\|_{L^2(S^1)}\|v\|_{L^2(S^1)}\|u_{2,r}\|_{L^{\infty}(S^1)}\\
 &\le C\|w_r\|_{L^2(S^1)}\|v\|_{L^2(S^1)}
 \le\frac{\varepsilon}{2}\|w_r\|^2_{L^2(S^1)}+C\|v\|^2_{L^2(S^1)}.
 \end{split}
\end{equation*}
Thus, with a constant $C=C(\varepsilon)>0$ we find
\begin{equation}\label{5.9}
 \frac{d}{dt}\|\nabla w\|^2_{L^2(B)}+\varepsilon\|w_r\|^2_{L^2(S^1)}\le C\|v\|^2_{L^2(S^1)}.
\end{equation}
Similarly, from \eqref{5.8} we can bound
\begin{equation}\label{5.10}
 \|w_t\|^2_{L^2(S^1)}\le C\|w_r\|^2_{L^2(S^1)}+C\|v\|^2_{L^2(S^1)}.
\end{equation}
Integrating over $0\le t\le T$ and observing that we have
\begin{equation*}
 \sup_{0\le t\le T}\|v(t)\|^2_{L^2(S^1)}\le\big(\int_0^T\|v_t(t)\|_{L^2(S^1)}dt\big)^2
 \le T\int_0^T\|v_t(t)\|^2_{L^2(S^1)}dt,
\end{equation*}
from \eqref{5.9} we first obtain 
\begin{equation*}
 \begin{split}
  \sup_{0\le t\le T}\|\nabla w(t)\|^2_{L^2(B)}+\varepsilon\|w_r\|^2_{L^2([0,T]\times S^1)}
  \le CT\sup_{0\le t\le T}\|v(t)\|^2_{L^2(S^1)}\le CT^2|v|^2_X,
\end{split}
\end{equation*}
which we may use together with \eqref{5.10} to bound
\begin{equation*}
 \begin{split}
 |w|^2_X&=\sup_{0\le t\le T}\|\nabla w(t)\|^2_{L^2(B)}+\|w_t\|^2_{L^2([0,T]\times S^1)}
 \le CT^2|v|^2_X.
\end{split}
\end{equation*}
For sufficiently small $T>0$ then our claim follows.
\end{proof}

Thus, by Banach's fixed point theorem, 
for any $\varepsilon>0$, any smooth $u_0\in H^{1/2}(S^1;N)$ there exists $T>0$ and a 
solution $u=u(t)\in V$ of the initial value problem \eqref{5.1}, \eqref{1.4}. We now show that the 
number $T>0$ may be chosen uniformly as $\varepsilon\downarrow 0$. 
Indeed, we have the following result.

\begin{lemma}\label{lemma5.4}
There exists a constant $C>0$ such that for any $k\ge 2$, any smooth 
$u_0\in H^{1/2}(S^1;N)$, and any $0<\varepsilon\le 1/2$ for the solution 
$u$ to \eqref{5.1} with $u(0)=u_0$ there holds 
\begin{equation*}
  \frac{d}{dt}\big(\|\nabla\partial^k_{\phi}u\|^2_{L^2(B)}\big)
  \le C(1+\|\nabla u\|^2_{L^2(B)}+\|\nabla\partial^k_{\phi}u\|_{L^2(B})^{k+3}.
\end{equation*}
\end{lemma}

\begin{proof}
Similar to the proof of Lemma \ref{lemma5.2}, for given $2\le k\in\N$ we compute
\begin{equation}\label{5.11}
 \begin{split}
 \frac12&\frac{d}{dt}\big(\|\nabla\partial^k_{\phi}u\|^2_{L^2(B)}\big)
 =(-1)^k\int_B\nabla\partial^{2k}_{\phi}u\nabla u_t\;dx\\
 &=(-1)^k(\partial^{2k}_{\phi}u_r,u_t)_{L^2(S^1)}
 =(-1)^{k+1}(\partial^{2k}_{\phi}u_r,
 (\varepsilon+d\pi_N(u))u_r)_{L^2(S^1)}\\
 &\le-\|d\pi_N(u)\partial^k_{\phi}u_r\|^2_{L^2(S^1)}-I,
 \end{split}
\end{equation}
where we now drop the term $\varepsilon\|\partial^k_{\phi}u_r\|^2_{L^2(S^1)}$ from \eqref{5.7}.
Again we split $I=\sum_{j=1}^k\Big({k\atop j}\Big)I_j$ with
\begin{equation*}
 \begin{split}
 I_j&=(\partial^k_{\phi}u_r,\partial^j_{\phi}(d\pi_N(u))
   \partial^{k-j}_{\phi}u_r)_{L^2(S^1)}\\
   &=(\nabla\partial^k_{\phi}u,\nabla(\partial^j_{\phi}(d\pi_N(u))
   \partial^{k-j}_{\phi}u_r))_{L^2(B)},
 \end{split}
\end{equation*}
but now we bound these terms as in the proof of Lemma \ref{lemma4.3} via
\begin{equation*}
 \begin{split}
 |I_j|&\le C\|\nabla\partial^k_{\phi}u\|_{L^2(B)}
 \big(\|\nabla\partial^j_{\phi}(d\pi_N(u))\partial^{k-j}_{\phi}u_r\|_{L^2(B)}
 +\|\partial^j_{\phi}(d\pi_N(u))\nabla\partial^{k-j}_{\phi}u_r\|_{L^2(B)}\big)\\
 &\le C\sum_{1\le j_i\le k+1,\,\Sigma_ij_i=k+2}\|\nabla\partial^k_{\phi}u\|_{L^2(B)}
 \|\Pi_i\nabla^{j_i}u\|_{L^2(B)}.
 \end{split}
\end{equation*}
Using that for any $k\ge 2$ by Sobolev's embedding 
$H^2(B)\hookrightarrow W^{1,4}\cap C^0(\bar{B})$ 
be can bound 
\begin{equation*}
 \begin{split}
  \sum_{1\le j_i\le k+1,\,\Sigma_ij_i=k+2}\|\Pi_i\nabla^{j_i}u\|_{L^2(B)}
  \le C(1+\|\nabla u\|_{L^2(B)}+\|\nabla^{k+1}u\|_{L^2(B)})^{k+2},
 \end{split}
\end{equation*}
and also using \eqref{4.3}, we obtain the claim.
\end{proof}

We now are able to conclude.
 
\begin{proposition}\label{prop5.5}
For any $k\ge 2$, any smooth $u_0\in H^{1/2}(S^1;N)$ there exists 
$T>0$ and a solution $u\in V$ to \eqref{1.3} with initial data $u(0)=u_0$.
\end{proposition}

\begin{proof}
In view of Lemma \ref{lemma5.4}, there exists a uniform number $T>0$ such that, with $V$ as 
defined above, for any $0<\varepsilon\le 1/2$ for there exists a solution 
$u_{\varepsilon}\in V$ to \eqref{5.1}. By definition of $V$, as $\varepsilon\downarrow 0$ 
suitably, we have $u_{\varepsilon}\to u$ weakly-$*$ in 
$L^{\infty}([0,T];H^{k+1}(B))\cap H^1(S^1\times[0,T])$. But this suffices to pass to the limit
$\varepsilon\downarrow 0$ in \eqref{5.1}, and $u\in V$ solves \eqref{1.3} with $u(0)=u_0$.  
\end{proof}

\begin{proof}[Proof of Theorem 1.1.i).]
By Proposition \ref{prop5.5} for any smooth $u_0\in H^{1/2}(S^1;N)$ and any $k\ge 2$ there 
exists $T>0$ and a solution $u\in V$ of \eqref{1.3}, \eqref{1.4} for $0<t<T$.
Alternatingly employing Propositions \ref{prop4.11} and \ref{prop4.6}, we then obtain
smoothness of $u$ for $0<t\le T$, including the final time $T$. (This argument later appears 
in more detail in Section \ref{Weak solutions} after Lemma \ref{lemma6.2}.)
Iterating, the solution $u$ 
may be extended smoothly until some maximal time $T_0$ where condition \eqref{4.4} 
ceases to hold. Uniqueness (even within a much larger class of competing functions) 
is established in Section \ref{Uniqueness}.
\end{proof}

\section{Weak solutions}\label{Weak solutions}
Given $u_0\in H^{1/2}(S^1;N)$, there are smooth functions
$u_{0k}\in H^{1/2}(S^1;N)$ with $u_{0k}\to u_0$ in $H^1(B)$ as $k\to\infty$.
Indeed, similar to an argument of Schoen-Uhlenbeck \cite{Schoen-Uhlenbeck-1982}, Theorem 3.1, 
with a standard mollifying sequence $(\rho_k)_{k\in\N}$ for the mollified functions
$v_{0k}:= u_0*\rho_k$ we have $dist_N(v_{0k})\to 0$ uniformly, and 
$u_{0k}:=\pi_N(v_{0k})\to u_0\in H^{1/2}(S^1;N)$ as $k\to\infty$. 

Let $u_k$ be the corresponding solutions of \eqref{1.4}
with initial data $u_k(0)=u_{0k}$, defined on a maximal time interval $[0,T_k[$, $k\in\N$.
We claim that each function $u_k$ can be smoothly extended 
to a uniform time interval $[0,T[$ for some $T>0$. To see this, we first establish 
the following non-concentration result.

\begin{lemma}\label{lemma6.1}
For any $\delta>0$ there exists a number $R>0$ and a time $T_0>0$ such that 
\begin{equation*}
   \sup_{z_0\in B,\,0<t<T_0}\int_{B_R(z_0)\cap B}|\nabla u_k(t)|^2dz<\delta
   \ \hbox{ for all }k\in\N.
\end{equation*} 
\end{lemma}

\begin{proof}
Given $\delta>0$, by absolute continuity of the Lebesgue integral and $H^1$-convergence
$u_{0k}\to u_0$ ($k\to\infty$) we can find $R>0$ such that 
\begin{equation*}
   \sup_{z_0\in B}\int_{B_{2R}(z_0)\cap B}|\nabla u_{0k}|^2dz<\delta
   \ \hbox{ for all }k\in\N.
\end{equation*} 
Choosing $T_0=\delta R$, by Lemma \ref{lemma2.2} then we have 
\begin{equation*}
   \sup_{z_0\in B,\,0<t<T_0}\int_{B_R(z_0)\cap B}|\nabla u_k(t)|^2dz
   <4\delta+C\delta E(u_{k0})<L\delta
\end{equation*} 
with a uniform constant $L>0$ for all $k\in\N$. The claim follows, if we replace 
$\delta$ with $\delta/L$.
\end{proof}

In view of Proposition \ref{prop3.3}, from Lemma \ref{lemma6.1} and Lemma \ref{lemma2.1}
we obtain the following bound for $u_k$ in $H^1(S^1)$.

\begin{lemma}\label{lemma6.2}
There exist a time $T_0>0$ and constants $C>0$, $C_0=C_0(E(u_0))>0$ such that 
\begin{equation*}
   \int_0^{T_0}\int_{S^1}|\partial_{\phi}u_k(t)|^2d\phi\,dt\le CE(u_{k0})\le C_0
   \ \hbox{ for all }k\in\N.
\end{equation*} 
\end{lemma} 

From Lemma \ref{lemma6.2} we obtain locally in time uniform smooth bounds for $(u_k)$ for $t>0$
by iteratively applying our previous regularity results. More precisely, Fatou's lemma 
and Lemma \ref{lemma6.2} first yield the bound 
\begin{equation*}
   \int_0^{T_0}\liminf_{k\to\infty}\big(\int_{S^1}|\partial_{\phi}u_k(t)|^2d\phi\big)dt
   \le C_0.
\end{equation*} 
Thus for almost every $0<t_0<T_0$ there holds 
\begin{equation*}
    \liminf_{k\to\infty}\int_{S^1}|\partial_{\phi}u_k(t_0)|^2d\phi<\infty.
\end{equation*} 
For any such $0<t_0<T_0$, if $\delta>0$ is sufficiently small, from Proposition \ref{prop4.2} 
with another appeal to Fatou's lemma we may conclude
\begin{equation*}
    \int_{t_0}^{T_0}\liminf_{k\to\infty}\int_B|\nabla\partial_{\phi}u_k|^2dz\,dt
    \le\liminf_{k\to\infty}\int_{t_0}^{T_0}\int_B|\nabla\partial_{\phi}u_k|^2dz\,dt\le C_1
\end{equation*}
for some $C_1>0$, so that now we even have
\begin{equation*}
    \liminf_{k\to\infty}\int_B|\nabla\partial_{\phi}u_k(t_1)|^2dz<\infty.
\end{equation*} 
for almost every $t_0<t_1<T_0$. Hence we may next invoke Proposition \ref{prop4.4} 
and \eqref{4.2} to obtain the bound
\begin{equation*}
    \liminf_{k\to\infty}
    \int_{t_1}^{T_0}\int_{\partial B}|\nabla\partial_{\phi}u_k|^2dz\,dt<\infty
\end{equation*}
for any such $t_0<t_1<T_0$, and Fatou's lemma gives that 
\begin{equation*}
    \liminf_{k\to\infty}\int_{\partial B}|\nabla\partial_{\phi}u_k(t_2)|^2d\phi<\infty
\end{equation*}
for almost every $t_1<t_2<T_0$. Now Proposition \ref{prop4.10} may be applied with $\varphi_0=1$, 
and we obtain 
\begin{equation*}
    \liminf_{k\to\infty}
    \int_{t_2}^{T_0}\int_B|\nabla\partial^2_{\phi}u_k|^2dz\,dt<\infty
\end{equation*}
for any such $t_1<t_2<T_0$. Another application of Fatou's lemma gives 
\begin{equation*}
    \liminf_{k\to\infty}\int_B|\nabla\partial^2_{\phi}u_k(t_3)|^2dz<\infty
\end{equation*}
for almost every $t_2<t_3<T_0$, and Proposition \ref{prop4.5} yields  
\begin{equation*}
    \liminf_{k\to\infty}
    \int_{t_3}^{T_0}\int_{\partial B}|\nabla\partial^2_{\phi}u_k|^2d\phi\,dt<\infty
\end{equation*}
for any such $t_2<t_3<T_0$. We may then iterate, using \eqref{3.7} and alternatingly employing 
Propositions \ref{prop4.11} and \ref{prop4.6} for $3\le k\in\N$, to find a subsequence 
$(u_k)$ satisfying uniform smooth bounds on $]t_0,T_0]$ for any $t_0>0$. Passing to the limit
$k\to\infty$ for this subsequence we obtain a weak solution to \eqref{1.3}, \eqref{1.4}
of energy-class in the following sense.

\begin{definition}\label{def6.3}
A function $u\in H^1([0,T_0]\times S^1;N)\cap L^{\infty}([0,T_0]; H^{1/2}(S^1;N))$ is a
weak solution of \eqref{1.3}, \eqref{1.4} of energy-class, if \eqref{1.3} is satisfied 
in the weak sense, that is, if there holds 
\begin{equation}\label{6.1}
  \begin{split}
  \int_0^{T_0}\int_{\partial B}&(u_t+d\pi_N(u)u_r)\cdot\varphi d\phi\,dt\\
  &=\int_0^{T_0}\int_{\partial B}u_t\cdot\varphi d\phi\,dt
  +\int_0^{T_0}\int_B\nabla u\cdot\nabla\big(d\pi_N(u)\varphi\big)dz\,dt=0
 \end{split}
\end{equation}
for all $\varphi\in C^{\infty}_c(S^1\times ]0,T_0[)$, and if there holds the energy 
inequality 
\begin{equation}\label{6.2}
  E(u(T))+\int_0^T\int_{\partial B}|u_t|^2d\phi\;dt\le E(u_0)
\end{equation}
for any $0<T<T_0$, with the initial data 
$u_0\in H^{1/2}(S^1;N)$ being attained in the sense of traces.
\end{definition}

We then may summarize our results, as follows. 

\begin{proposition}\label{prop6.4}
For any $u_0\in H^{1/2}(S^1;N)$ there exists $T_0>0$ and a weak solution $u$ to \eqref{1.3},
\eqref{1.4} on $[0,T_0]$ of energy-class, which is smooth for $t>0$.
\end{proposition}

\begin{proof}
For any open $U\subset S^1\times ]0,T_0[$ we have uniform smooth bounds for $u_k$ on $U$;
thus a suitable sub-sequence $u_k\to u$ smoothly locally as $k\to\infty$. 
The equation \eqref{6.1} follows from the corresponding identites for $u_k$.

Moreover, \eqref{6.2} follows from the energy identity, Lemma \ref{lemma2.1}, for $u_k$ in view of
$H^1$-convergence $u_{0k}\to u_0$ as well as weak lower semi-continuity of the energy and of 
the $L^2$-norm.

Finally, with error $o(1)\to 0$ as $k\to\infty$ for $0<t<T_0$ we can estimate 
\begin{equation*}
  \begin{split}
  \|&u(t)-u_0\|^2_{L^2(\partial B)}\le\|u_k(t)-u_{0k}\|^2_{L^2(\partial B)}+o(1)\\
  &\le\Big(\int_0^t\|\partial_tu_k(t')\|_{L^2(\partial B)}dt'\Big)^2+o(1)
  \le t\int_0^t\|\partial_tu_k(t')\|^2_{L^2(\partial B)}dt'+o(1)\\
  &\le tE(u_0)+o(1)\to 0\ \hbox{ as } t\downarrow 0,
 \end{split}
\end{equation*}
and $u(t)\to u_0$ weakly in $H^{1/2}(S^1;N)\cap H^1(B;\R^n)$ as $t\downarrow 0$. In fact,
by \eqref{6.2} we then even have strong convergence. 
\end{proof}

\section{Uniqueness}\label{Uniqueness}
With the help of the tools developed in Section \ref{A regularity estimate} 
we can show uniqueness of partially regular weak energy-class solutions as in 
Proposition \ref{prop6.4}. 

\begin{theorem}\label{thm7.1} 
Let $u_0\in H^{1/2}(S^1;N)$. Suppose $u$ and $v$ both are 
weak energy-class solutions of \eqref{1.3}, \eqref{1.4} on $[0,T_0]$ for some $T_0>0$ with 
initial data $u_0$, and suppose that $u$ and $v$ are smooth for $t>0$. Then $u=v$.
\end{theorem}

\begin{proof}
Using the identity \eqref{3.2} for $u$ and $v$, respectively, for the function $w=u-v$ 
for almost every $0<t<T_0$ we have 
\begin{equation}\label{7.1}
  \begin{split}
  &\partial_tw+\partial_rw=\nu(u)\partial_r(dist_N(u))-\nu(v)\partial_r(dist_N(v))\\
  &\quad=(\nu(u)-\nu(v))\partial_r(dist_N(u))+\nu(v)\partial_r(dist_N(u)-dist_N(v))
 \end{split}
\end{equation}
on $\partial B=S^1$.
From equation \eqref{3.5}, moreover, we obtain
\begin{equation}\label{7.2}
 \begin{split} 
  |\Delta(dist_N(u)&-dist_N(v))|=|\nabla u\cdot d\nu(u)\nabla u-\nabla v\cdot d\nu(v)\nabla v|\\
  &\le C(|w||\nabla u|^2+(|\nabla u|+|\nabla v|)|\nabla w|)\hbox{ in }B.
 \end{split}
\end{equation}
Observing that 
\begin{equation*}
  |dist_N(u)-dist_N(v)|\le C|w|,
\end{equation*}
upon multiplying \eqref{7.2} with the function $(dist_N(u)-dist_N(v))\in H_0^1(B)$,
integrating by parts, and using Young's inequality, for any $\varepsilon>0$ we obtain 
\begin{equation}\label{7.3}
  \begin{split}
 \|\nabla&(dist_N(u)-dist_N(v))\|^2_{L^2(B)}\\
 &\le C\int_B(|w|^2|\nabla u|^2+(|\nabla u|+|\nabla v|)|\nabla w||w|)dz\\
 &\le\varepsilon\|\nabla w\|^2_{L^2(B)}
 +C(\varepsilon)\|w\|^2_{L^4(B)}(\|\nabla u\|^2_{L^4(B)}+\|\nabla v\|^2_{L^4(B)}).
  \end{split}
\end{equation}
On the other hand, for any $0<t_0<T\le T_0$, multiplying the equation \eqref{7.1} with $w$ 
and integrating by parts on $S^1\times [t_0,T]$, upon letting $t_0\downarrow 0$ we find 
\begin{equation*}
  \begin{split}
 \sup_{0<t<T}&\|w(t)\|^2_{L^2(\partial B)}+\int_0^T\int_B|\nabla w|^2dz\,dt
 \le C\int_0^T\int_{\partial B}(\partial_tw+\partial_rw)w\,d\phi\,dt\\
 &=C\int_0^T\int_{\partial B}w(\nu(u)-\nu(v))\partial_r(dist_N(u))d\phi\,dt\\
 &\quad+C\int_0^T\int_{\partial B}w\,\nu(v)\partial_r(dist_N(u)-dist_N(v))\,d\phi\,dt
 =:C\int_0^T(I+II)dt.
  \end{split}
\end{equation*}
We first estimate the term
\begin{equation*}
  \begin{split}
   I&=I(t)=\int_{\partial B}w(\nu(u)-\nu(v))\partial_r(dist_N(u))\,d\phi\\
   &=\int_B\nabla\big(w(\nu(u)-\nu(v))\big)\nabla(dist_N(u))\,dz\\
   &\quad+\int_Bw(\nu(u)-\nu(v))\Delta(dist_N(u))\,dz.
   \end{split}
\end{equation*}
Using
\begin{equation*}
 \begin{split}
 |\nabla\big(w(\nu(u)-\nu(v))\big)|&\le C|\nabla w||w|
 +|w\big((d\nu(u)-d\nu(v))\nabla u+d\nu(v)\nabla w\big)|\\
 &\le C(|\nabla w||w|+|w|^2|\nabla u|)
 \end{split}
\end{equation*}
we can bound 
\begin{equation*}
  \begin{split}
   |\int_B&\nabla\big(w(\nu(u)-\nu(v))\big)\nabla(dist_N(u))dz|
   \le C\int_B|(\nabla w||w|+|w|^2|\nabla u|)|\nabla u|dz\\
 &\le \varepsilon\|\nabla w\|^2_{L^2(B)}+C(\varepsilon)\|w\|^2_{L^4(B)}\|\nabla u\|^2_{L^4(B)}
  \end{split}
\end{equation*}
for each $t$. Also using \eqref{3.5}, we can moreover estimate 
\begin{equation*}
  \begin{split}
   |\int_B&w(\nu(u)-\nu(v))\Delta(dist_N(u))\,dz|
   \le C\|w\|^2_{L^4(B)}\|\nabla u\|^2_{L^4(B)}
  \end{split}
\end{equation*}
for almost every $0<t<T$ to obtain
\begin{equation*}
  \begin{split}
   |I|\le \varepsilon\|\nabla w\|^2_{L^2(B)}
   +C(\varepsilon)\|w\|^2_{L^4(B)}\|\nabla u\|^2_{L^4(B)}.
  \end{split}
\end{equation*}

Similarly, we estimate the term
\begin{equation*}
  \begin{split}
   II&=II(t)=\int_{\partial B}w\,\nu(v)\partial_r((dist_N(u)-dist_N(v))\,d\phi\\
   &=\int_B\nabla(w\nu(v))\nabla(dist_N(u)-dist_N(v))dz\\
   &\quad+\int_Bw\,\nu(v)\Delta(dist_N(u)-dist_N(v))\,dz.
  \end{split}
\end{equation*}
Noting that with \eqref{7.3} we can bound
\begin{equation*}
  \begin{split}
   |\int_B&\nabla(w\nu(v))\nabla(dist_N(u)-dist_N(v))dz|\\
   &\le C(\|\nabla w\|_{L^2(B)}+\|w\nabla v\|_{L^2(B)})\|\nabla(dist_N(u)-dist_N(v))\|_{L^2(B)}\\
   &\le\varepsilon\|\nabla w\|^2_{L^2(B)}
   +C(\varepsilon)\|w\|^2_{L^4(B)}(\|\nabla u\|^2_{L^4(B)}+\|\nabla v\|^2_{L^4(B)})
   \end{split}
\end{equation*}
and that with \eqref{7.2} we have
\begin{equation*}
  \begin{split}
  |\int_B&w\nu(v)\Delta(dist_N(u)-dist_N(v))\,dz|\\
  &\le C\int_B(|w|^2|\nabla u|^2+|w||\nabla w|(|\nabla u|+|\nabla v|))\,dz\\
  &\le\varepsilon\|\nabla w\|^2_{L^2(B)}
  +C(\varepsilon)\|w\|^2_{L^4(B)}(\|\nabla u\|^2_{L^4(B)}+\|\nabla v\|^2_{L^4(B)})
  \end{split}
\end{equation*}
we find the estimate
\begin{equation*}
  \begin{split}
   |II|\le\varepsilon\|\nabla w\|^2_{L^2(B)}+C(\varepsilon)\|w\|^2_{L^4(B)}
   (\|\nabla u\|^2_{L^4(B)}+\|\nabla v\|^2_{L^4(B)})
  \end{split}
\end{equation*}
for almost every $0<t<T$.

But Sobolev's embedding $H^{1/2}(B)\hookrightarrow L^4(B)$ and Fourier expansion give
the bound 
\begin{equation*}
   \|w\|^2_{L^4(B)}\le C\|w\|^2_{H^{1/2}(B)}\le C\|w\|^2_{L^2(\partial B)}
\end{equation*}
and similar bounds for $\nabla u$ as well as $\nabla v$.
Moreover, since by the energy inequality \eqref{6.2} we have $u(t),v(t)\to u_0$ strongly
in $H^1(B)$ as $t\downarrow 0$, there exist a radius $0<R\le1/2$ and a time $0<T<T_0$ 
such that condition \eqref{3.13} in Proposition \ref{prop3.3} holds true on
$[0,T]$ for both $u$ and $v$, allowing to bound   
\begin{equation*}
  \begin{split}
   \int_0^T&\|\nabla u(t)\|^2_{L^4(B)}dt
   \le C\int_0^T\|\nabla u(t)\|^2_{L^2(\partial B)}dt
   \le C\int_0^T\|\partial_{\phi}u(t)\|^2_{L^2(\partial B)}dt\\
   &\le C\int_0^T\int_{\partial B}|u_t|^2d\phi\,dt+C(R)TE(u_0)\le C(R)(1+T_0)E(u_0)
  \end{split}
\end{equation*}
with the help of \eqref{3.7}, and similarly for $|\nabla v|$. 
Choosing $\varepsilon=1/4$, for sufficiently small $0<T<T_0$ by absolute continuity of
the integral we thus can estimate
\begin{equation*}
 \begin{split}
 &\sup_{0<t<T}\|w(t)\|^2_{L^2(\partial B)}+\int_0^T\int_B|\nabla w|^2dz\,dt\\
 &\quad\le\frac12\|\nabla w\|^2_{L^2(B\times[0,T])}
 +C\sup_{0<t<T}\|w(t)\|^2_{L^2(\partial B)}
 \int_0^T(\|\nabla u\|^2_{L^4(B)}+\|\nabla v\|^2_{L^4(B)})dt\\
 &\quad\le\frac12\Big(\sup_{0<t<T}\|w(t)\|^2_{L^2(\partial B)}
 +\int_0^T\int_B|\nabla w|^2dz\,dt\Big),
 \end{split}
\end{equation*}
and it follows that $w=0$, as claimed.
\end{proof}

\begin{proof}[Proof of Theorem \ref{thm1.2}]
Existence for short time and uniqueness of a partially regular weak solution 
to \eqref{1.3}, \eqref{1.4} for given data $u_0\in H^{1/2}(S^1;N)$ follow from 
Proposition \ref{prop6.4} and Theorem \ref{thm7.1}, respectively. 
Since by Proposition \ref{prop6.4} our weak solution is smooth for $t>0$, the remaining 
assertions follow from Theorem \ref{thm1.1}. 

Note that at any blow-up time $T_{i-1}$, $i\ge 1$, of the flow as in Theorem \ref{thm1.1}.ii) 
there exists a unique weak limit $u_i=\lim_{t\uparrow T_{i-1}}u(t)\in H^{1/2}(S^1;N)$,
and we may uniquely continue the flow using Proposition \ref{prop6.4}.
\end{proof}

\section{Blow-up}\label{Blow-up}
Preparing for the proof of part ii) of Theorem \ref{thm1.1} suppose now that for the solution
constructed in part i) of that theorem there holds $T_0<\infty$.
Then, as we shall see in more detail below, by the results in Section \ref{Higher regularity} 
condition \eqref{4.4} must be violated for $T=T_0$ and there exist $\delta>0$ and points 
$z_k\in B$ as well as radii $r_k\downarrow 0$ as $k\to\infty$ such that for suitable 
$t_k\uparrow T_0$ there holds 
\begin{equation*}
   \int_{B_{r_k}(z_k)\cap B}|\nabla u(t_k)|^2dz
   =\sup_{z_0\in B,\,t\le t_k}\int_{B_{r_k}(z_0)\cap B}|\nabla u(t)|^2dz=\delta.
\end{equation*}
We may later choose a smaller constant $\delta>0$, if necessary. Moreover, for later use
from now on we consider local concentrations in the sense that for some $z_0\in B$ and some 
fixed radius $r_0>0$ for a sequence of points $z_k\in B$ with $z_k\to z_0$ and radii 
$r_k\downarrow 0$ for suitable $t_k\uparrow T_0$ as $k\to\infty$ there holds
\begin{equation*}
   \int_{B_{r_k}(z_k)\cap B}|\nabla u(t_k)|^2dz
   =\sup_{z'\in B_{r_0}(z_0),\,t\le t_k}\int_{B_{r_k}(z')\cap B}|\nabla u(t)|^2dz=\delta.
\end{equation*}

Scale 
\begin{equation*}
  u_k(z,t)=u(z_k+r_kz,t_k+r_kt)
\end{equation*}
for 
\begin{equation*}
  z\in\Omega_k=\{z;\;z_k+r_kz\in B\},\ t\in I_k=\{t;\; 0\le t_k+tr_k<T_0\}.
\end{equation*}
Note that then there holds
\begin{equation}\label{8.1}
 \begin{split}
   \int_{B_1(0)\cap\Omega_k}&|\nabla u_k(0)|^2dz\\
   &=\sup_{z_k+r_kz'\in B_{r_0}(z_0),-t_k/r_k\le t<0}
   \int_{B_1(z')\cap\Omega_k}|\nabla u_k(t)|^2dz=\delta.
 \end{split}
\end{equation}
Passing to a sub-sequence we may assume that the domains $\Omega_k$ exhaust a limit domain
$\Omega_{\infty}\subset\R^2$, which either is the whole space $\R^2$ or a half-space $H$.

By the energy inequality Lemma \ref{lemma2.1} for $t\in I_k$
there holds
\begin{equation}\label{8.2}
   \int_{\Omega_k}|\nabla u_k(t)|^2dz=\int_{B}|\nabla u(t_k+r_kt)|^2dz\le 2E(u_0),
\end{equation}
and for any $t_0<0$ and sufficiently large $k\in\N$ we have
\begin{equation}\label{8.3}
 \begin{split}
  \int_{t_0}^0&\int_{\partial\Omega_k}|\partial_tu_k|^2ds\;dt
  =\int_{t_0}^0\int_{\partial\Omega_k}|d\pi_N(u_k)\partial_{\nu_k}u_k|^2ds\;dt\\
  &=\int_{t_k+r_kt_0}^{t_k}\int_{\partial B}|u_t|^2d\phi\;dt
  \le\int_{t_k+r_kt_0}^{T_0}\int_{\partial B}|u_t|^2d\phi\;dt\to 0
 \end{split}
\end{equation}
as $k\to\infty$, where $ds$ is the element of length and where $\nu_k$ is the outward unit 
normal along $\partial\Omega_k$.
Expressing the harmonic functions $\partial_tu_k(t)$ in Fourier series for each $t<0$,
it then also follows that $\partial_tu_k\to 0$ locally in 
$L^2$ on $\Omega_{\infty}\times]-\infty,0[$. Finally, again using the fact that 
$u_k(t)$ for each $t$ is harmonic, by the maximum principle we have the uniform bound
$|u_k|\le\sup_{p\in\Gamma}|p|$ as well as uniform smooth bounds locally away from the 
boundary of $\Omega_{\infty}$.
 
Hence we may assume that as $k\to\infty$ we have $u_k\to u_{\infty}$ weakly locally in 
$H^1$ on $\Omega_{\infty}\times]-\infty,0[$, where $u_{\infty}(z,t)=u_{\infty}(z)$ is 
independent of time, harmonic, and bounded. Moreover, we have smooth convergence away from 
$\partial\Omega_{\infty}$. 
Thus, if we assume that $\Omega_{\infty}=\R^2$ by \eqref{8.1} it follows that  
\begin{equation*}
   \int_{B_1(0)}|\nabla u_{\infty}|^2dz=\delta.
\end{equation*}
But any function $v\colon\R^2\to\R$ which is bounded and harmonic must be constant,
which rules out this possibility. Hence $\Omega_{\infty}$ can only be a half-space. 

After a suitable rotation of the domain $B$ and shift of coordinates in $\R^2\cong\C$
we may then assume that $z_k=(0,-y_k)$ with $1-y_k\le Mr_k$ for some $M\in\N$ and that 
$\Omega_{\infty}=\{(x,y);\;y>y_0\}$ for some $y_0$. Finally, replacing $r_k>0$ with $(M+1)r_k$
and $z_k$ with $z_k=(0,-1)$, if necessary, we may assume that 
$\Omega_k\subset\R^2_+=\{(x,y);\;y>0\}$ is the ball of radius $1/r_k$ around the point 
$(0,1/r_k)$ with $0\in\partial\Omega_k$, 
while from \eqref{8.1} with a uniform number $L\in\N$ we have
\begin{equation}\label{8.4}
   L\int_{B_1(0)\cap\Omega_k}|\nabla u_k(0)|^2dz
   \ge L\delta\ge\sup_{|z'|\le r_0/r_k,-t_k/r_k\le t<0}
   \int_{B_1(z')\cap\Omega_k}|\nabla u_k(t)|^2dz
\end{equation}
for any $k\in\N$. Let $\Phi_k\colon\R^2_+\to\Omega_k$ be the conformal maps given by 
\begin{equation*}
    \Phi_k(z)=\frac{2z}{2-ir_kz},\ z\in\R^2_+,\ k\in\N,
\end{equation*}
with $\Phi_k\to id$ locally uniformly on $\R^2\cong\C$ as $k\to\infty$. 

Let $v_k=u_k\circ\Phi_k$, $k\in\N$. By
conformal invariance of the Dirichlet energy, from \eqref{8.2} for any $t$ we have 
\begin{equation}\label{8.5}
   \int_{\R^2_+}|\nabla v_k(t)|^2dz=\int_{\Omega_k}|\nabla u_k(t)|^2dz\le 2E(u_0),
\end{equation}
and by \eqref{8.4} with a uniform number $L_1\in\N$ 
there holds
\begin{equation}\label{8.6}
   L_1\int_{B^+_2(0)}|\nabla v_k(0)|^2dz\ge L_1\delta
   \ge\sup_{|z'|\le r_0/r_k,-t_k/r_k\le t<0}\int_{B^+_1(z')}|\nabla v_k(t)|^2dz,
\end{equation}
where $B^+_r(z)=B_r(z)\cap\R^2_+$ for any $r>0$ and any $z=(x,y)\in\R^2$. 
Moreover, from \eqref{8.3} for any $t_0<0$ and any $R>0$ for the integral over 
$]-R,R\,[\times\{0\}\subset\partial\R^2_+$ we obtain  
\begin{equation}\label{8.7}
 \begin{split}
  \int_{t_0}^0\int_{-R}^R&|\partial_tv_k|^2dx\;dt\\
  &\le C\int_{t_0}^0\int_{-R}^R|d\pi_N(v_k)\partial_yv_k|^2dx\;dt\to 0\ \hbox{ as }k\to\infty, 
 \end{split}
\end{equation}
and $\partial_tv_k\to 0$ locally in $L^2$ on $\overline{\R^2_+}\times]-\infty,0[$. 
In addition, from our choice of $(u_k)$ it follows that $v_k\to v_{\infty}$ weakly
locally in $H^1$ on $\overline{\R^2_+}\times]-\infty,0[$ as $k\to\infty$, 
where $v_{\infty}(z,t)=:w_{\infty}(z)$ is harmonic and bounded. 

For a suitable sequence of times $t_0<s_k<0$, 
we then also have locally weak convergence $w_k:=v_k(s_k)\to w_{\infty}$ in $H^1$ on
$\overline{\R^2_+}$ and, in addition,
\begin{equation}\label{8.8}
  d\pi_N(w_k)\partial_yw_k\to 0\ \hbox{ in } L_{loc}^2(\partial\R^2_+)\ \hbox{ as }k\to\infty.
\end{equation}
Thus, for sufficiently small $\delta>0$ by Proposition \ref{prop3.3}, applied to the functions
$w_k\circ\Psi$, where $\Psi\colon B\to\R^2_+$ is a suitable conformal map, we also have uniform 
local $L^2$-bounds for $\partial_xw_k$ on $\partial\R^2_+$, and we may assume that 
$w_k\to w_{\infty}$ locally uniformly and weakly 
locally in $H^1$ on $\partial\R^2_+$ as $k\to\infty$. Since $w_k$ is harmonic,
we then also have locally strong $H^1$-convergence $w_k\to w_{\infty}$ on $\overline{\R^2_+}$.

To see that $w_{\infty}$ is non-constant, let $\varphi_k=\varphi_{z_0,4r_k}$, $k\in\N$.
Integrating the identity \eqref{2.1} from the proof of Lemma \ref{lemma2.2} in time, 
with error $o(1)\to 0$ and suitable numbers $\varepsilon_k\downarrow 0$ as $k\to\infty$ in
view of \eqref{8.3} we find  
\begin{equation}\label{8.9}
 \begin{split}
  \frac12\big|\int_B&|\nabla u(t_k)|^2\varphi_k^2dz
  -\int_B|\nabla u(t_k+r_ks_k)|^2\varphi_k^2dz\big|\\
  &\le\int_{t_k+r_ks_k}^{t_k}\int_{\partial B}|u_t|^2\varphi_k^2d\phi\,dt
  +2\int_{t_k+r_ks_k}^{t_k}\int_B|u_t\nabla u\varphi_k\nabla\varphi_k|dz\,dt\\
  &\le o(1)+8\varepsilon_kr_k\int_{t_k+r_ks_k}^{t_k}\int_B|\nabla u|^2|\nabla\varphi_k|^2dz\,dt\\
  &\qquad+(8\varepsilon_kr_k)^{-1}\int_{t_k+r_ks_k}^{t_k}\int_B|u_t|^2\varphi_k^2dz\,dt.
 \end{split}
\end{equation}
With the help of \eqref{2.2} and \eqref{8.3} for suitable $\varepsilon_k\downarrow 0$
we can bound 
\begin{equation*}
 \begin{split}
  (8\varepsilon_kr_k)^{-1}\int_{t_k+r_ks_k}^{t_k}\int_B|u_t|^2\varphi_k^2dz\,dt
  \le C\varepsilon_k^{-1}\int_{t_k+r_ks_k}^{t_k}\int_{\partial B}|u_t|^2dz\,dt\to 0.
 \end{split}
\end{equation*}
Since for any choice $t_0<s_k<0$ we also can estimate 
\begin{equation*}
 \begin{split}
  8\varepsilon_kr_k\int_{t_k+r_ks_k}^{t_k}\int_B|\nabla u|^2|\nabla\varphi_k|^2dz\,dt
  \le C\varepsilon_k|t_0|E(u_0))\to 0,
 \end{split}
\end{equation*}
from \eqref{8.9} and \eqref{8.6} it follows that with error $o(1)\to 0$ as $k\to\infty$ 
we have
\begin{equation}\label{8.10}
 \begin{split}
    L_1\int_{B^+_4(0)}&|\nabla w_k|^2dz+o(1)=L_1\int_{B^+_4(0)}|\nabla v_k(s_k)|^2dz+o(1)\\
   &\ge L_1\int_{B}|\nabla u(t_k+r_ks_k)|^2\varphi_k^2dz+o(1)
   \ge L_1\int_{B}|\nabla u(t_k)|^2\varphi_k^2dz\\
   &\ge L_1\int_{B^+_2(0)}|\nabla v_k(0)|^2dz\ge L_1\delta.
 \end{split}
\end{equation}

Finally, in view of locally uniform convergence $w_k\to w_{\infty}$ and weak local 
$L^2$-convergence of the traces $\nabla w_k\to\nabla w_{\infty}$ on $\partial\R^2_+$,
we may pass to the limit $k\to\infty$ in \eqref{8.8} to conclude that 
\begin{equation}\label{8.11}
         d\pi_N(w_{\infty})\partial_yw_{\infty}=0\ \hbox{ on }\partial\R^2_+.
\end{equation}

Since $w_{\infty}$ is harmonic, the Hopf differential 
\begin{equation*}
  f=|\partial_xw_{\infty}|^2-|\partial_yw_{\infty}|^2
  -2i\partial_xw_{\infty}\cdot\partial_yw_{\infty}
\end{equation*}
defines a holomorphic function $f\in L^1(\R^2_+,\C)$.
Moreover, $w_{\infty}\in H^{3/2}_{loc}(\R^2_+)$ 
with trace $\nabla w_{\infty}\in L^2_{loc}(\partial\R^2_+)$;
thus also the trace of $f$ is well-defined on $\partial\R^2_+$. By \eqref{8.11} now the trace 
of $f$ is real-valued; thus $f\equiv c$ for some constant $c\in\R$.
But $\nabla w_{\infty}\in L^2(\R^2_+)$; hence $f\in L^1(\R^2_+)$. It follows that $c=0$,
and $w_{\infty}$ is conformal. 

With stereographic projection 
$\Phi\colon B\to\R^2_+$ from a point $z_0\in\partial B$
define the map $\bar{u}=w_{\infty}\circ\Phi\in H^{1/2}(S^1;N)$. 
By conformal invariance, $\bar{u}$ again is harmonic with finite Dirichlet integral 
and satisfies \eqref{1.6} 
on $\partial B\setminus\{z_0\}$; 
since the point $\{z_0\}$ has vanishing $H^1$-capacity, $\bar{u}$ then is stationary 
in the sense of \cite{Gruter-et-al-1981}.
Moreover, $\bar{u}$ is conformal. For such mappings, smooth regularity on $\bar{B}$ 
was shown by Gr\"uter-Hildebrandt-Nitsche \cite{Gruter-et-al-1981}; 
thus condition \eqref{1.6} holds everywhere on $\partial B$ in the pointwise sense,
and $\bar{u}$ parametrizes a minimal surface of 
finite area supported by $N$ which meets $N$ orthogonally along its boundary. 

\begin{proof}[Proof of Theorem \ref{thm1.1}.ii)]
For given smooth data $u_0\in H^{1/2}(S^1;N)$ let $u$ be the unique solution 
to \eqref{1.3}, \eqref{1.4} guaranteed by part i) of the theorem, and suppose that the
maximal time of existence $T_0<\infty$. Then condition \eqref{4.4} must fail as 
$t\uparrow T_0$; else from Propositions \ref{prop4.11} and \ref{prop4.6} we obtain 
smooth bounds for $u(t)$ as $t\uparrow T_0$ and there exists a smooth trace 
$u_1=\lim_{t\uparrow T_0}u(t)$. But by the first part of the theorem there is a smooth 
solution to the initial value problem for \eqref{1.3} with initial data $u_1$ at time $T_0$,
and this solution extends the original solution $u$ to an interval $[0,T_1[$ for some 
$T_1>T_0$, contradicting maximality of $T_0$.

Let $z^{(i)}\in B$, $1\le i\le i_0$, such that for some number $\delta>0$ and 
suitable $t_k^{(i)}\uparrow T_0$, $z_k^{(i)}\to z^{(i)}$, $r_k^{(i)}\to 0$ as $k\to\infty$ 
there holds
\begin{equation*}
   \liminf_{k\to\infty}\int_{B_{r_k^{(i)}}(z_k^{(i)})\cap B}|\nabla u(t_k^{(i)})|^2dz\ge\delta.
\end{equation*}
By the argument following \eqref{8.9} thus for a suitable sequence of radii 
$0<r_k^{(0)}\to 0$ such that $r_k^{(i)}/r_k^{(0)}\to 0$ as well as $(T_0-t_k^{(i)})/r_k^{(0)}\to 0$
then with error $o(1)\to 0$ as $k\to\infty$ there holds
\begin{equation*}
   \int_{B_{r_k^{(0)}}(z^{(i)})\cap B}|\nabla u(t)|^2dz+o(1)
   \ge\int_{B_{r_k^{(i)}}(z_k^{(i)})\cap B}|\nabla u(t_k^{(i)})|^2dz\ge\delta.
\end{equation*}
for all $T_0-r_k^{(0)}<t<T_0$, uniformly in $1\le i\le i_0$. 
For sufficiently large $k\in\N$ such that $r_k^{(0)}<\inf_{i<j}|z^{(i)}-z^{(j)}|/4$ it 
follows that $i_0\le E(u_0)/\delta$, and we may fix $r_0>0$ and redefine $t_k^{(i)}$, 
$r_k^{(i)}$, and $z_k^{(i)}$, if necessary, such that for each $1\le i\le i_0$ there holds 
\begin{equation*}
   \int_{B_{r_k^{(i)}}(z_k^{(i)})\cap B}|\nabla u(t_k^{(i)})|^2dz
   =\sup_{z'\in B_{r_0}(z^{(i)}),\,0<t\le t_k^{(i)}}\int_{B_{r_k^{(i)}}(z')\cap B}|\nabla u(t)|^2dz=\delta.
\end{equation*}
Moreover, we may assume that $\delta<\delta_0$, as defined in Proposition \ref{prop3.1}.
The characterization of the concentration points as in Theorem \ref{thm1.2}.ii) via 
solutions $\bar{u}^{(i)}$ of \eqref{1.6} then follows from our above analysis.

In addition, Corollary \ref{cor3.2} yields the uniform lower bound 
\begin{equation*}
   \lim_{r_0\downarrow 0}\liminf_{t\uparrow T}\int_{B_{r_0}(z^{(i)})\cap B}|\nabla u(t)|^2dz
   \ge 2E(\bar{u}^{(i)})\ge 2\delta_0^2
\end{equation*}
for the concentration energy quanta, which gives the claimed upper bound for the total number of 
concentration points. 

Finally, with the help of Proposition \ref{prop4.11} we can smoothly extend the solution $u$ to 
$B\setminus\{z^{(1)},\dots,z^{(i_0)}\}$ at time $t=T_0$.
\end{proof}

\section{Asymptotics}\label{Asymptotics}
Suppose next that the solution $u$ to \eqref{1.3}, \eqref{1.4} exists for all time $0<t<\infty$.
Then $u$ either concentrates for suitable $t_k\uparrow\infty$ in the sense that 
condition \eqref{4.4} does not hold true uniformly in time, or $u$ satisfies uniform smooth 
bounds, as shown in Section \ref{Higher regularity}. 

In the latter case, the claim made in Theorem \ref{thm1.1}.iii) easily follows. 

\begin{proposition}\label{prop9.1}
Suppose that for any $\delta>0$ there exists $R>0$ such that condition \eqref{4.4} holds 
true for all $0<t<\infty$. Then there exists a smooth solution $u_{\infty}\in H^{1/2}(S^1;N)$ of 
\eqref{1.6} such that $u(t)\to u_{\infty}$ smoothly as $t\to\infty$ suitably, 
and $u_{\infty}$ parametrizes a minimal surface of finite area supported by $N$ which meets 
$N$ orthogonally along its boundary.
\end{proposition}
 
\begin{proof} 
For sufficiently small $\delta>0$, for any $j\in\N$ by iterative reference to 
Propositions \ref{prop4.2}, \ref{prop4.4} - \ref{prop4.6}, 
and \ref{prop4.10}, \ref{prop4.11}, respectively, as in Section \ref{Weak solutions} 
we can find constants $C_j>0$ such that $\|u(t)\|_{H^j(B)}\le C_j$ for all $t>1$, Moreover, 
by the energy inequality Lemma \ref{lemma2.1} for a suitable sequence $t_k\to\infty$
there holds $u_t(t_k)\to 0$ in $L^2(\partial B)$ as $k\to\infty$. Then for any $j\in\N$ a 
subsequence $u(t_k)\to u_{\infty}$ in $H^j(B)$, and a diagonal subsequence converges smoothly, 
where $u_{\infty}$ solves \eqref{1.6}. By the argument after \eqref{8.11} in 
Section \ref{Blow-up} then $u_{\infty}$ is conformal and $u_{\infty}$ parametrizes a minimal
surface with free boundary on $N$ which meets $N$ orthogonally along its boundary.
\end{proof} 

In the remaining case that for some $\delta>0$ condition \eqref{4.4} fails to hold, there exists 
a sequence $t_k\uparrow\infty$ and points $z^{(1)},\dots,z^{(i_0)}$ such that for sequences 
$z_k^{(i)}\to z^{(i)}$, radii $r_k^{(i)}\to 0$ as $k\to\infty$ there holds
\begin{equation*}
   \liminf_{k\to\infty}\int_{B_{r_k^{(i)}}(z_k^{(i)})\cap B}|\nabla u(t_k)|^2dz\ge\delta,\
   1\le i\le i_0.
\end{equation*}
By Lemma \ref{lemma2.1} there holds the a-priori bound $i_0\le E(u_0)/\delta$ for the number 
of concentration points. By the argument leading to \eqref{8.10} then for a suitable number
$0<r_0\le\inf_{i<j}|z^{(i)}-z^{(j)}|/4$ with error $o(1)\to0$ as $k\to\infty$ 
and with some constant $L\in\N$ for all $1\le i\le i_0$ there holds
\begin{equation*}
 \begin{split}
   L\int_{B_{2r_k^{(i)}}(z_k^{(i)})\cap B}&|\nabla u(t_k)|^2dz+o(1)\\
   &\ge\sup_{z_0\in B_{r_0}(z_k^{(i)}),\,t_k-r_0\le t\le t_k}
   \int_{B_{r_k^{(i)}}(z_0)\cap B}|\nabla u(t)|^2dz\ge\delta.
 \end{split}
\end{equation*}
Fixing any index $1\le i\le i_0$ and renaming $z_k^{(i)}=:z_k$, $r_k^{(i)}=:r_k$, we then scale 
\begin{equation*}
  u_k(z,t)=u(z_k+r_kz,t_k+r_kt), \ z\in\Omega_k=\{z;z_k+r_kz\in B\},\ -t_k/r_k\le t\le 0,
\end{equation*}
as before and observe that for any $t_0<0$ there holds
\begin{equation}\label{9.1}
 \begin{split}
  \int_{t_0}^0&\int_{\partial\Omega_k}|\partial_tu_k|^2ds\;dt
  =\int_{t_0}^0\int_{\partial\Omega_k}|d\pi_N(u_k)\partial_{\nu_k}u_k|^2ds\;dt\\
  &=\int_{t_k+r_kt_0}^{t_k}\int_{\partial B}|u_t|^2d\phi\;dt
  \le\int_{t_k+r_kt_0}^{\infty}\int_{\partial B}|u_t|^2d\phi\;dt\to 0
 \end{split}
\end{equation}
as $k\to\infty$, where $\nu_k$ is the outward unit normal along $\partial\Omega_k$.
Just as in Section \ref{Blow-up} for suitable $t_0<s_k<0$ we then obtain local uniform 
and $H^1$-convergence of a subsequence of the conformally rescaled maps 
$w_k=u_k(s_k)\circ\Phi_k\in H^1_{loc}(\R^2_+)$ to a smooth, harmonic and conformal limit 
$w_{\infty}$ with finite energy and continuously mapping $\partial\R^2_+$ to $N$, inducing 
a solution $\bar{u}_{\infty}=w_{\infty}\circ\Phi\in H^{1/2}(S^1;N)$ of \eqref{1.6}
corresponding to a minimal surface with free boundary on $N$.
This ends the proof of Theorem \ref{thm1.1}.iii) 

\section{Appendix}
In this section, for the convenience of the reader we derive two interpolation inequalities 
that play a crucial role in our arguments.

Let $v\in H^1(B)$, and let $\varphi_{z_i,r}$ as above such that the collection of 
balls $B_r(z_i)$, $1\le i\le i_0$ covers $\bar{B}$ with at most $L$ balls $B_{2r}(z_i)$
overlapping at any $z\in B$, with $L\in\N$ independent of $r>0$. 
We may assume $r<1/8$ so that for any $1\le i\le i_0$ there is a pair of 
orthogonal vectors $e_{1,i}$, $e_{2,i}$ such that for any $z\in B_r(z_i)$ there holds 
$z+se_{1,i}+te_{2,i}\in B$ for any $0\le s,t\le 2r$. After a rotation of coordinates, 
we may assume that $e_{1,i}=(1,0)$, $e_{2,i}=(0,1)$ are the standard basis vectors. 
Writing $\varphi$ for $\varphi_{z_i,r}$ for any $z=(x,y)\in B_r(z_i)$,
by arguing as Ladyzhenskaya \cite{Ladyzhenskaya-1969}, using that 
\begin{equation*}
  (v^2\varphi)(x+2r,y)=0=(v^2\varphi)(x,y+2r),
\end{equation*}
then we can estimate 
\begin{equation}\label{A.1}
 \begin{split}
  v^4(z)&=|(v^2\varphi)(z)|^2\le\int_0^{2r}|\partial_x(v^2\varphi)(x+s,y)|ds
  \cdot\int_0^{2r}|\partial_y(v^2\varphi)(x,y+t)|dt\\
  &\le\int_{\{s;(s,y)\in B\}}|\partial_x(v^2\varphi)(s,y)|ds
  \cdot\int_{\{t;(x,t)\in B\}}|\partial_y(v^2\varphi)(x,t)|dt,
 \end{split}
\end{equation}
and with the help of Fubini's theorem we find 
\begin{equation*}
 \begin{split}
  &\int_{B_r(z_i)}|v|^4dz\le\int_B|v|^4\varphi^2dz
  \le\int_{-\infty}^{\infty}\big(\int_{\{x;(x,y)\in B\}}|(v^2\varphi)(x,y)|^2dx\big)dy\\
  &\le\int_{-\infty}^{\infty}\int_{\{s;(s,y)\in B\}}|\partial_x(v^2\varphi)(s,y)|ds\,dy
  \cdot\int_{-\infty}^{\infty}\int_{\{t;(x,t)\in B\}}|\partial_y(v^2\varphi)(x,t)|dt\,dx\\
  &\le\big(\int_B|\nabla(v^2\varphi)|dz\big)^2
  \le\big(\int_B(2|\nabla v||v\varphi|+v^2|\nabla\varphi|)dz\big)^2\\
  &\le C\big(\int_{B_{2r}(z_i)}|\nabla v|^2dz+r^{-2}\int_{B_{2r}(z_i)}v^2dz\big)
  \int_{B_{2r}(z_i)}v^2dz.
 \end{split}
\end{equation*}
Fixing $r=1/5$ and summing over $1\le i\le i_0$ with an absolute constant $C>0$ 
we obtain the bound
\begin{equation}\label{A.2}
 \begin{split}
 \|v\|^4_{L^4(B)}
 \le C\|v\|^2_{H^1(B)}\|v\|^2_{L^2(B)} 
 \end{split}
\end{equation}
for any $v\in H^1(B)$.


\begin{thebibliography}{1}

\bibitem{Brezis-Gallouet-1980} Br\'ezis, H.; Gallouet, T.:
{\em Nonlinear Schr\"odinger evolution equations}, Nonlinear Anal. 4 (1980), no. 4, 677-681. 

\bibitem{Brendle-2002}
Brendle, Simon: {\em Curvature flows on surfaces with boundary},
Math. Ann. 324 (2002), no. 3, 491-519.
 
\bibitem{Brezis-Wainger-1980} Br\'ezis, Haim; Wainger, Stephen:
{\em A note on limiting cases of Sobolev embeddings and convolution inequalities},
Comm. Partial Differential Equations 5 (1980), no. 7, 773-789.
 
\bibitem{Caffarelli-Silvestre-2007}
Caffarelli, Luis; Silvestre, Luis: {\em An extension problem related to the fractional Laplacian},
Comm. Partial Differential Equations 32 (2007), no. 7-9, 1245-1260. 

\bibitem{Chang-Ding-Ye-1992}
  Chang, Kung-Ching; Ding, Wei Yue; Ye, Rugang:
  {\em Finite-time blow-up of the heat flow of harmonic maps from surfaces},
  J. Differential Geom. 36 (1992), no. 2, 507-515.
  
\bibitem{Chang-Liu-2005}
Chang, Kung-Ching; Liu, Jia-Quan: {\em Boundary flow for the minimal surfaces in
$\R^n$ with Plateau boundary condition}, Proc. Roy. Soc. Edinburgh Sect. A 135 (2005), 
no. 3, 537-562. 

\bibitem{Chang-Liu-2004}
Chang, Kung-Ching; Liu, Jia-Quan: {\em An evolution of minimal surfaces with Plateau condition},
Calc. Var. Partial Differential Equations 19 (2004), no. 2, 117-163. 

\bibitem{Chang-Liu-2003a}
Chang, Kung-Ching; Liu, Jia-Quan: 
{\em Heat flow for the minimal surface with Plateau boundary condition},
Acta Math. Sin. (Engl. Ser.) 19 (2003), no. 1, 1-28. 

\bibitem{Chang-Liu-2003b}
Chang, Kung-Ching; Liu, Jia-Quan: 
{\em Another approach to the heat flow for Plateau problem},
J. Differential Equations 189 (2003), no. 1, 46-70.

\bibitem{Courant-1937}
Courant, Richard:
{\em Plateau's problem and Dirichlet's principle},
Ann. of Math. (2) 38 (1937), no. 3, 679-724.
 
\bibitem{Da Lio-2015}
Da Lio, Francesca: {\em Compactness and bubble analysis for $1/2$-harmonic maps},
Ann. Inst. H. Poincar\'e Anal. Non Lin\'eaire 32 (2015), no. 1, 201-224. 

\bibitem{Da Lio-Martinazzi-Riviere-2015}
Da Lio, Francesca; Martinazzi, Luca; Rivi\`ere, Tristan: 
{\em Blow-up analysis of a nonlocal Liouville-type equation}, 
Anal. PDE 8 (2015), no. 7, 1757-1805.

\bibitem{Da Lio-Pigati-2020}
Da Lio, Francesca; Pigati, Alessandro: {\em Free boundary minimal surfaces: 
a nonlocal approach}, Ann. Sc. Norm. Super. Pisa Cl. Sci. (5) 20 (2020), no. 2, 437-489. 

\bibitem{Da Lio-Riviere-2011}  
Da Lio, Francesca; Rivi\`ere, Tristan: {\em Three-term commutator estimates 
and the regularity of $1/2$-harmonic maps into spheres}, Anal. PDE 4 (2011), no. 1, 149-190.

\bibitem{Da Lio-Schikorra-2017}  
Da Lio, Francesca; Schikorra, Armin: {\em On regularity theory for $n/p$-harmonic maps 
into manifolds}, Nonlinear Anal. 165 (2017), 182-197.

\bibitem{Douglas-1931}
Douglas, Jesse: {\em Solution of the problem of Plateau},
Trans. Amer. Math. Soc. 33 (1931), no. 1, 263-321. 

\bibitem{Eells-Sampson-1964}
Eells, James, Jr.; Sampson, J. H.: {\em Harmonic mappings of Riemannian manifolds},
Amer. J. Math. 86 (1964), 109-160.

\bibitem{Freire-1995}
Freire, Alexandre: {\em Uniqueness for the harmonic map flow from surfaces to general targets},
Comment. Math. Helv. 70 (1995), no. 2, 310-338. 

\bibitem{Freire-1996}
  Freire, Alexandre: {\em Correction to: "Uniqueness for the harmonic map flow from
  surfaces to general targets''}, Comment. Math. Helv. 71 (1996), no. 2, 330-337.

\bibitem{Gehrig-2020}
Gehrig, Manuela Iris: {\em Prescribed curvature on the boundary of the disc}, 
Dissertation, ETH Zurich, 2020, https://doi.org/10.3929/ethz-b-000445412.

\bibitem{Gruter-et-al-1981}
Gr\"uter, Michael; Hildebrandt, Stefan; Nitsche, Johannes C. C.:
{\em On the boundary behavior of minimal surfaces with a free boundary which are not 
minima of the area}, Manuscripta Math. 35 (1981), no. 3, 387-410.
 
\bibitem{Hyder-et-al-2021}
Hyder, Ali; Segatti, Antonio; Sire, Yannick; Wang, Changyou:
{\em Partial regularity of the heat flow of half-harmonic maps and applications to 
harmonic maps with free boundary}, arXiv:2111.14171.

\bibitem{Imbusch-Struwe-1999}
Imbusch, Cordula; Struwe, Michael: {\em Variational principles for minimal surfaces}, 
Topics in nonlinear analysis, 477-498, Progr. Nonlinear Differential Equations Appl., 35, 
Birkh\"auser, Basel, 1999.
  
\bibitem{Jost-1981}
Jost, J\"urgen: {\em Univalency of harmonic mappings between surfaces}, 
J. Reine Angew. Math. 324 (1981), 141-153.

\bibitem{Jost-Struwe-1990}
Jost, J.; Struwe, M.: {\em Morse-Conley theory for minimal surfaces of varying topological type},
Invent. Math. 102 (1990), no. 3, 465-499. 

\bibitem{Ladyzhenskaya-1969}
Ladyzhenskaya, O. A.: {\em  The mathematical theory of viscous incompressible flow},
Second English edition, revised and enlarged Translated from the Russian by Richard
A. Silverman and John Chu, Mathematics and its Applications, Vol. 2, Gordon and Breach,
Science Publishers, New York-London-Paris, 1969.

\bibitem{Lenzmann-Schikorra-2020}
  Lenzmann, Enno; Schikorra, Armin: {\em Sharp commutator estimates via harmonic extensions},
  Nonlinear Anal. 193 (2020), 111375, 37 pp.

\bibitem{Mazowiecka-Schikorra-2018}
Mazowiecka, Katarzyna; Schikorra, Armin: {\em Fractional div-curl quantities and 
applications to nonlocal geometric equations}, J. Funct. Anal. 275 (2018), no. 1, 1-44.

\bibitem{Millot-Sire-2015}
Millot, Vincent; Sire, Yannick: {\em On a fractional Ginzburg-Landau equation and 
$1/2$-harmonic maps into spheres}, Arch. Ration. Mech. Anal. 215 (2015), no. 1, 125-210.

\bibitem{Morse-1937} 
Morse, Marston: {\em Functional topology and abstract variational theory},
Ann. of Math. (2) 38 (1937), no. 2, 386-449. 

\bibitem{Morse-Tompkins-1939}
Morse, Marston; Tompkins, C.: 
{\em The existence of minimal surfaces of general critical types},
Ann. of Math. (2) 40 (1939), no. 2, 443-472.

\bibitem{Moser-2011}
  Moser, Roger: {\em Intrinsic semiharmonic maps},
  J. Geom. Anal. 21 (2011), no. 3, 588-598. 

\bibitem{Rado-1930}
Rad\'o, Tibor: {\em On Plateau's problem}, 
Ann. of Math. (2) 31 (1930), no. 3, 457-469.

\bibitem{Riviere-1993}
  Rivi\`ere, Tristan: {\em Le flot des applications faiblement harmoniques en dimension deux},
  published in ``Applications harmoniques entre vari\'et\'es'', Th\`ese de l’universit\'e
  Paris 6, 1993.

\bibitem{Rupflin-2017}
Rupflin, Melanie: {\em  Teichm\"uller harmonic map flow from cylinders},
Math. Ann. 368 (2017), no. 3-4, 1227-1276. 
 
\bibitem{Rupflin-Schrecker-2018}
Rupflin, Melanie; Schrecker, Matthew R. I.: 
{\em Analysis of boundary bubbles for almost minimal cylinders},
Calc. Var. Partial Differential Equations 57 (2018), no. 5, Paper No. 121.
 
\bibitem{Rupflin-Topping-2019}
Rupflin, Melanie; Topping, Peter M.: 
{\em Global weak solutions of the Teichm\"uller harmonic map flow into general targets}, 
Anal. PDE 12 (2019), no. 3, 815-842.

\bibitem{Schikorra-2012}
Schikorra, Armin: {\em Regularity of $n/2$-harmonic maps into spheres},
J. Differential Equations 252 (2012), no. 2, 1862-1911.

\bibitem{Schoen-Uhlenbeck-1982}
Schoen, Richard; Uhlenbeck, Karen: {\em A regularity theory for harmonic maps},
J. Differential Geometry 17 (1982), no. 2, 307-335.
 
\bibitem{Schoen-Yau-1978}
Schoen, Richard; Yau, Shing Tung: {\em On univalent harmonic maps between surfaces},
Invent. Math. 44 (1978), no. 3, 265-278.

\bibitem{Struwe-1984}
Struwe, Michael: {\em On a critical point theory for minimal surfaces spanning a wire in $\R^n$},
J. Reine Angew. Math. 349 (1984), 1-23.

\bibitem{Struwe-1985}
Struwe, Michael: {\em On the evolution of harmonic mappings of Riemannian surfaces},
Comment. Math. Helv. 60 (1985), no. 4, 558-581.

\bibitem{Struwe-1986} 
Struwe, Michael: {\em A Morse theory for annulus-type minimal surfaces},
J. Reine Angew. Math. 368 (1986), 1-27.

\bibitem{Struwe-1988}
Struwe, Michael: {\em Plateau's problem and the calculus of variations},
Mathematical Notes, 35. Princeton University Press, Princeton, NJ, 1988. 

\bibitem{Topping-2002}
Topping, Peter: {\em Reverse bubbling and nonuniqueness in the harmonic map flow},
Int. Math. Res. Not. 2002, no. 10, 505-520.
 
\bibitem{Tromba-1985}
Tromba, Anthony Joseph: 
{\em Degree theory on oriented infinite-dimensional varieties and the Morse number
of minimal surfaces spanning a curve in $\R^n$. I, $n\ge 4$},
Trans. Amer. Math. Soc. 290 (1985), no. 1, 385-413.

\bibitem{Tromba-1984}
Tromba, Anthony Joseph: 
{\em Degree theory on oriented infinite-dimensional varieties and the Morse number of minimal surfaces spanning a curve in $\R^n$. II, $n=3$},
Manuscripta Math. 48 (1984), no. 1-3, 139-161. 

\bibitem{Wettstein-2022}
Wettstein, Jerome: 
{\em Uniqueness and Regularity of the Fractional Harmonic
Gradient Flow in $S^{n-1}$}, Nonlinear Analysis 214 (2022).

\bibitem{Wettstein-2021a}
Wettstein, Jerome: 
{\em Existence, Uniqueness and Regularity of the Fractional
Harmonic Gradient Flow in General Target Manifolds}, arxiv 2109.11458.

\bibitem{Wettstein-2021b}
Wettstein, Jerome: 
{\em Half-Harmonic Gradient Flow: Aspects of a Non-Local Geometric PDE}, arXiv:2112.08846.

\end{thebibliography}
\end{document}